\documentclass[a4paper, british]{amsart}

\usepackage{amssymb}
\usepackage{babel}
\usepackage{enumitem}
\usepackage{hyperref}
\usepackage[utf8]{inputenc}
\usepackage{newunicodechar}
\usepackage{varioref}
\usepackage[arrow,curve,matrix]{xy}

\usepackage{colortbl}
\usepackage{graphicx}
\usepackage{tikz}

\usepackage{mathrsfs}

\definecolor{linkred}{rgb}{0.7,0.2,0.2}
\definecolor{linkblue}{rgb}{0,0.2,0.6}

\numberwithin{figure}{section}

\usepackage[hyperpageref]{backref}

\setdescription{labelindent=\parindent, leftmargin=2\parindent}
\setitemize[1]{labelindent=\parindent, leftmargin=2\parindent}
\setenumerate[1]{labelindent=0cm, leftmargin=*, widest=iiii}

\newunicodechar{α}{\ensuremath{\alpha}}
\newunicodechar{β}{\ensuremath{\beta}}
\newunicodechar{χ}{\ensuremath{\chi}}
\newunicodechar{δ}{\ensuremath{\delta}}
\newunicodechar{ε}{\ensuremath{\varepsilon}}
\newunicodechar{Δ}{\ensuremath{\Delta}}
\newunicodechar{η}{\ensuremath{\eta}}
\newunicodechar{γ}{\ensuremath{\gamma}}
\newunicodechar{Γ}{\ensuremath{\Gamma}}
\newunicodechar{ι}{\ensuremath{\iota}}
\newunicodechar{κ}{\ensuremath{\kappa}}
\newunicodechar{λ}{\ensuremath{\lambda}}
\newunicodechar{Λ}{\ensuremath{\Lambda}}
\newunicodechar{μ}{\ensuremath{\mu}}
\newunicodechar{ω}{\ensuremath{\omega}}
\newunicodechar{Ω}{\ensuremath{\Omega}}
\newunicodechar{π}{\ensuremath{\pi}}
\newunicodechar{φ}{\ensuremath{\phi}}
\newunicodechar{Φ}{\ensuremath{\Phi}}
\newunicodechar{ψ}{\ensuremath{\psi}}
\newunicodechar{Ψ}{\ensuremath{\Psi}}
\newunicodechar{ρ}{\ensuremath{\rho}}
\newunicodechar{σ}{\ensuremath{\sigma}}
\newunicodechar{Σ}{\ensuremath{\Sigma}}
\newunicodechar{τ}{\ensuremath{\tau}}
\newunicodechar{θ}{\ensuremath{\theta}}
\newunicodechar{Θ}{\ensuremath{\Theta}}

\newunicodechar{⊗}{\ensuremath{\otimes}}
\newunicodechar{→}{\ensuremath{\to}}
\newunicodechar{⨯}{\ensuremath{\times}}
\newunicodechar{∪}{\ensuremath{\cup}}
\newunicodechar{∩}{\ensuremath{\cap}}
\newunicodechar{⊇}{\ensuremath{\supseteq}}
\newunicodechar{⊃}{\ensuremath{\supset}}
\newunicodechar{⊆}{\ensuremath{\subseteq}}
\newunicodechar{⊂}{\ensuremath{\subset}}
\newunicodechar{≥}{\ensuremath{\geq}}
\newunicodechar{≤}{\ensuremath{\leq}}
\newunicodechar{∈}{\ensuremath{\in}}
\newunicodechar{◦}{\ensuremath{\circ}}
\newunicodechar{°}{\ensuremath{^\circ}}
\newunicodechar{…}{\ifmmode\mathellipsis\else\textellipsis\fi}
\newunicodechar{·}{\ensuremath{\cdot}}

\newunicodechar{¹}{\ensuremath{^1}}
\newunicodechar{²}{\ensuremath{^2}}
\newunicodechar{³}{\ensuremath{^3}}

\DeclareFontFamily{OMS}{rsfs}{\skewchar\font'60}
\DeclareFontShape{OMS}{rsfs}{m}{n}{<-5>rsfs5 <5-7>rsfs7 <7->rsfs10 }{}
\DeclareSymbolFont{rsfs}{OMS}{rsfs}{m}{n}
\DeclareSymbolFontAlphabet{\scr}{rsfs}
\DeclareSymbolFontAlphabet{\scr}{rsfs}

\DeclareMathOperator{\Aut}{Aut}
\DeclareMathOperator{\codim}{codim}
\DeclareMathOperator{\coker}{coker}

\DeclareMathOperator{\Id}{Id}
\DeclareMathOperator{\Image}{Image}
\DeclareMathOperator{\img}{img}
\DeclareMathOperator{\Pic}{Pic}
\DeclareMathOperator{\rank}{rank}

\DeclareMathOperator{\red}{red}
\DeclareMathOperator{\reg}{reg}

\DeclareMathOperator{\sing}{sing}

\DeclareMathOperator{\supp}{supp}
\DeclareMathOperator{\tor}{tor}

\newcommand{\sB}{\scr{B}}

\newcommand{\sE}{\scr{E}}

\newcommand{\sF}{\scr{F}}
\newcommand{\sG}{\scr{G}}
\newcommand{\sH}{\scr{H}}
\newcommand{\sHom}{\scr{H}\negthinspace om}
\newcommand{\sI}{\scr{I}}
\newcommand{\sJ}{\scr{J}}

\newcommand{\sL}{\scr{L}}

\newcommand{\sO}{\scr{O}}

\newcommand{\sT}{\scr{T}}

\newcommand{\bC}{\mathbb{C}}

\newcommand{\bN}{\mathbb{N}}

\newcommand{\bP}{\mathbb{P}}
\newcommand{\bQ}{\mathbb{Q}}

\newcommand{\bZ}{\mathbb{Z}}

\theoremstyle{plain}   
\newtheorem{thm}{Theorem}[section]

\newtheorem{cor}[thm]{Corollary}
\newtheorem{defn}[thm]{Definition} 
\newtheorem{fact}[thm]{Fact}
\newtheorem{lem}[thm]{Lemma}

\newtheorem{prop}[thm]{Proposition}

\theoremstyle{remark}
\newtheorem{assumption}[thm]{Assumption} 
\newtheorem{asswlog}[thm]{Assumption w.l.o.g.}
\newtheorem{claim}[thm]{Claim}
\newtheorem{c-n-d}[thm]{Claim and Definition}
\newtheorem{consequence}[thm]{Consequence}
\newtheorem{construction}[thm]{Construction}

\newtheorem{notation}[thm]{Notation}
\newtheorem{obs}[thm]{Observation}
\newtheorem{rem}[thm]{Remark}

\newtheorem*{rem-nonumber}{Remark}

\numberwithin{equation}{thm}

\setlist[enumerate]{label=(\thethm.\arabic*), before={\setcounter{enumi}{\value{equation}}}, after={\setcounter{equation}{\value{enumi}}}}

\newcommand{\into}{\hookrightarrow}

\newcommand{\wtilde}{\widetilde}
\newcommand{\what}{\widehat}

\def\clap#1{\hbox to 0pt{\hss#1\hss}}

\newcommand\CounterStep{\addtocounter{thm}{1}\setcounter{equation}{0}}

\newcommand{\factor}[2]{\left. \raise 2pt\hbox{$#1$} \right/\hskip -2pt\raise -2pt\hbox{$#2$}}

\makeatletter
\let\saveqed\qed
\renewcommand\qed{%
   \ifmmode\displaymath@qed
   \else\saveqed
   \fi}
\makeatother
\newcommand{\Preprint}[1]{#1}
\newcommand{\Publication}[1]{}

\newcommand{\subversionInfo}{}
\newcommand{\svnid}[1]{}
\newcommand{\approvals}[1]{}

\ifdefined\isPublication
   \usepackage{times}
\else
   \usepackage{mathpazo}
\fi

\DeclareMathOperator{\Br}{Branch}
\DeclareMathOperator{\ch}{ch} 
\DeclareMathOperator{\divisor}{div}

\DeclareMathOperator{\Gal}{Gal}
\DeclareMathOperator{\GL}{GL}
\DeclareMathOperator{\Iso}{Iso}
\DeclareMathOperator{\Link}{Link}

\title[Étale fundamental groups, flat sheaves, and quotients of Abelian varieties]{Étale fundamental groups of Kawamata log terminal spaces, flat sheaves, and quotients of Abelian varieties}

\author{Daniel Greb}
\author{Stefan Kebekus}
\author{Thomas Peternell}

\thanks{All three authors were supported in part by the DFG-Forschergruppe 790
  ``Classification of Algebraic Surfaces and Compact Complex Manifolds''.
  Daniel Greb gratefully acknowledges the support of the
  Baden--Württemberg--Stiftung through the ``Eliteprogramm für
  Postdoktorandinnen und Postdoktoranden'' and by the Institute of Mathematics
  at Albert-Ludwigs-Universität Freiburg.  Stefan Kebekus acknowledges the
  support through a joint fellowship of the Freiburg Institute of Advanced
  Studies (FRIAS) and the University of Strasbourg Institute for Advanced Study
  (USIAS)}

\address{Daniel Greb, Essener Seminar für Algebraische Geometrie und Arithmetik\\Fakultät für Ma\-the\-matik\\Universität Duisburg-Essen\\
45117 Essen\\ Germany}
\email{\href{mailto:daniel.greb@uni-due.de}{daniel.greb@uni-due.de}}
\urladdr{\href{http://www.esaga.uni-due.de/daniel.greb}{http://www.esaga.uni-due.de/daniel.greb}}

\address{Stefan Kebekus, Mathematisches Institut, Albert-Ludwigs-Universität
  Freiburg, Eckerstraße 1, 79104 Freiburg im Breisgau, Germany and University of
  Strasbourg Institute for Advanced Study (USIAS), Strasbourg, France}
\email{\href{mailto:stefan.kebekus@math.uni-freiburg.de}{stefan.kebekus@math.uni-freiburg.de}}
\urladdr{\href{http://home.mathematik.uni-freiburg.de/kebekus}{http://home.mathematik.uni-freiburg.de/kebekus}}

\address{Thomas Peternell, Mathematisches Institut, Universität Bayreuth,
  95440~Bayreuth, Germany}
\email{\href{mailto:thomas.peternell@uni-bayreuth.de}{thomas.peternell@uni-bayreuth.de}}
\urladdr{\href{http://www.staff.uni-bayreuth.de/~btm109/peternell/pet.html}{http://www.staff.uni-bayreuth.de/$\sim$btm109/peternell/pet.html}}



\hypersetup{
  pdfauthor={Daniel Greb, Stefan Kebekus, Thomas Peternell},
  pdftitle={Étale covers of klt spaces and their smooth loci},
  pdfkeywords={Minimal Model Program, Algebraic Fundamental Group, KLT Singularities},
  pdfstartview={Fit},
  pdfpagelayout={TwoColumnRight},
  pdfpagemode={UseOutlines},
  bookmarks,
  colorlinks,
  linkcolor=linkblue,
  citecolor=linkred,
  urlcolor=linkred
}

\date{\today}

\begin{document}

\maketitle

\tableofcontents

%
%
\svnid{$Id: 01.tex 873 2015-07-28 13:14:15Z kebekus $}

\section{Introduction}
\subversionInfo
\label{sec:intro}
\approvals{Greb & --- \\ Kebekus & yes \\ Peternell & ---}

Working with a singular complex algebraic variety $X$, one is often interested
in comparing the set of finite étale covers of $X$ with that of its smooth locus
$X_{\reg}$.  More precisely, one may ask the following.
\begin{quote}
  What are the obstructions to extending finite étale covers of $X_{\reg}$ to
  $X$?  How do the étale fundamental groups of $X$ and of its smooth locus
  differ?
\end{quote}

We answer these questions for projective varieties $X$ with Kawamata log
terminal (klt) singularities, a class of varieties that is important in the
Minimal Model Program.  The main result, Theorem~\ref{thm:main1}, asserts that
in any infinite tower of finite Galois morphisms over a klt variety, where all
morphisms are étale in codimension one, almost every morphism is étale.  In a
certain sense, this result can be seen as saying that the difference between the
sets of étale covers of $X$ and of $X_{\reg}$ is small in case $X$ is klt.  As
an immediate application, we construct in Theorem~\ref{thm:imm1} a finite
covering $\wtilde X → X$, étale in codimension one, such that the étale
fundamental groups of $\wtilde X$ and of its smooth locus $\wtilde X_{\reg}$
agree.  Further direct applications concern global bounds for the index of
$\bQ$-Cartier divisors on klt spaces, the existence of a quasi-étale
``simultaneous index-one cover'' where all $\bQ$-Cartier divisors are Cartier,
and a finiteness result for étale fundamental groups of varieties of weak log
Fano-type.

As a first major application, we obtain an extension theorem for flat vector
bundles on klt varieties: after passing to a finite cover, étale in codimension
one, any flat holomorphic bundle, defined on the smooth part of a klt variety
extends across the singularities, to a flat bundle that is defined on the whole
space.  As a consequence, we show that every variety with terminal singularities
and with vanishing first and second Chern class is a finite
quotient of an Abelian variety, with a quotient map that is étale in codimension
one.  In the smooth case, this is a classical result, which follows from the
existence of Kähler-Einstein metrics on Ricci-flat compact Kähler manifolds, due
to Yau: every Ricci-flat compact Kähler manifold $X$ with $c_2(X) = 0$ is an
étale quotient of a compact complex torus.  As a further application, we verify
a conjecture of Nakayama and Zhang concerning varieties admitting polarised
endomorphisms.  Building on their results, we obtain a decomposition theorem
describing the structure of these varieties.

\subsection{Simplified version of the main result}
\approvals{Greb & yes \\ Kebekus & yes \\ Peternell & ---}

Our main result, Theorem~\ref{thm:main1}, is quite general and its formulation
is therefore somewhat involved.  For many applications the following special
case suffices.

\begin{thm}\label{thm:mainSpec}
  Let $X$ be a normal, complex, quasi-projective variety.  Assume that there
  exists a $\bQ$-Weil divisor $Δ$ such that $(X, Δ)$ is Kawamata log terminal
  (klt).  Assume we are given a sequence of finite, surjective morphisms of
  normal varieties that are étale in codimension one,
  \begin{equation}\label{eq:ghh}
    \xymatrix{
      X = Y_0 & \ar[l]_(.4){γ_1} Y_1 & \ar[l]_{γ_2} Y_2 & \ar[l]_{γ_3} Y_3 & \ar[l]_{γ_4} \cdots.
    }
  \end{equation}
  If the composed morphisms $γ_1 ◦ \cdots ◦ γ_i : Y_i → X$ are Galois for every
  $i ∈ \bN^+$, then all but finitely many of the morphisms $γ_i$ are étale.
\end{thm}

\begin{rem}[Galois morphisms]
  In Theorem~\ref{thm:mainSpec} and throughout this paper, Galois morphisms are
  assumed to be finite and surjective, but need not be étale; see
  Definition~\ref{def:Galois}.  The statement of Theorem~\ref{thm:mainSpec} does
  not continue to hold when one drops the Galois assumption.
  Section~\ref{subsect:Kummer} discusses an example that illustrates the
  problems in the non-Galois setting.
\end{rem}

\begin{rem}
  In accordance with the notation of Kollár-Kovács \cite{KK10} one might call
  varieties $X$ ``potentially klt'' if they admit a divisor $Δ$ that makes the
  pair $(X,Δ)$ klt.  For sake of brevity, we will later call quasi-finite,
  surjective morphisms which are étale in codimension one ``quasi-étale''; see
  Definition~\ref{defn:quasietale}, a terminology which was first introduced by
  F.~Catanese in \cite{MR2344354}.
\end{rem}

\begin{rem}[Purity of branch locus]\label{rem:main}
  By purity of the branch locus, the assumption that all morphisms $γ_i$ of
  Theorem~\ref{thm:mainSpec} are étale in codimension one can also be formulated
  in one of the following, equivalent ways.
  \begin{enumerate}
  \item All morphisms $γ_1 ◦ \cdots ◦ γ_i$ are étale over the smooth locus of
    $Y_0$.
  \item Given any $i ∈ \bN^+$, then $γ_i$ is étale over the smooth locus of
    $Y_{i-1}$.
  \end{enumerate}
\end{rem}

\subsection{Direct Applications}
\approvals{Greb & yes \\ Kebekus & yes \\ Peternell & ---}

Theorems~\ref{thm:mainSpec} and \ref{thm:main1} have a large number of immediate
consequences.  As a first direct application, we show that every klt space
admits a quasi-étale cover whose étale fundamental group equals that of its
smooth locus.

\begin{thm}[Extension of étale covers from the smooth locus of klt spaces]\label{thm:imm1}
  Let $X$ be a normal, complex, quasi-projective variety.  Assume that there
  exists a $\bQ$-Weil divisor $Δ$ such that $(X, Δ)$ is klt.  Then, there exists
  a normal variety $\wtilde X$ and a finite, surjective Galois morphism $γ:
  \wtilde X → X$, étale in codimension one, such that the following equivalent
  conditions hold.
  \begin{enumerate}
  \item\label{il:mc1} Any finite, étale cover of $\wtilde X_{\reg}$ extends to a
    finite, étale cover of $\wtilde X$.

  \item\label{il:mc2} The natural map $\what{ι}_* : \what{π}_1(\wtilde X_{\reg})
    → \what{π}_1(\wtilde X)$ of étale fundamental groups induced by the
    inclusion of the smooth locus, $ι : \wtilde X_{\reg} → \wtilde X$, is an
    isomorphism.
  \end{enumerate}
\end{thm}

In fact, somewhat more general statements are true;
cf.~Section~\ref{ssec:cormain1gen2}.  To avoid any potential for confusion, we
briefly recall the characterisation of the étale fundamental group of a complex
variety.

\begin{fact}[\protect{Étale fundamental group, \cite[§~5 and references there]{Milne80}}]
  If $Y$ is any complex algebraic variety, then the étale fundamental group
  $\what{π}_1(Y)$ is isomorphic to the profinite completion of the topological
  fundamental group of the associated complex space $Y^{an}$.  \qed
\end{fact}

\begin{rem}[Reformulation of Theorem~\ref{thm:imm1} in terms of étale fundamental groups]\label{rem:p1hat}
  Although it might seem natural, Theorem~\ref{thm:imm1} does \emph{not} imply
  that the kernel of the natural map $ι_* : \what{π}_1(X_{\reg}) →
  \what{π}_1(X)$ is finite.  A counterexample is discussed in
  Section~\vref{ssec:GZ}.  We do not know whether Theorem~\ref{thm:imm1} can be
  expressed solely in terms of the étale fundamental groups
  $\what{π}_1(X_{\reg})$ and $\what{π}_1(X)$.
\end{rem}

\begin{rem}[Canonical choice of minimal $\wtilde X$]
  It is natural to ask whether there exists a canonical choice of $\wtilde X$,
  uniquely determined by a suitable minimality property.  This is not the case.
  We will show in Section~\ref{sec:noMinTildeX} that in general no ``minimal
  cover'' exists.
\end{rem}

The following local variant of Theorem~\ref{thm:imm1} considers coverings of
neighbourhoods of a given point, rather than coverings of the full space.

\begin{thm}[Local version of Theorem~\ref{thm:imm1}]\label{thm:imm1local}
  Let $X$ be a normal, complex, quasi-projective variety.  Assume that there
  exists a $\bQ$-Weil divisor $Δ$ such that $(X, Δ)$ is klt.  Let $p ∈ X$ be any
  closed point.  Then, there exists a Zariski-open neighbourhood $X°$ of $p ∈
  X$, a normal variety $\wtilde X°$ and a finite, surjective Galois morphism $γ
  : \wtilde X° → X°$, étale in codimension one, such that the following holds:
  given any Zariski-open neighbourhood $U = U(p) ⊆ X°$ with preimage $\wtilde U
  := γ^{-1}(U)$, then $\what{π}_1(\wtilde U_{\reg}) \cong \what{π}_1(\wtilde
  U)$.  Equivalently, any finite, étale cover of $\wtilde U_{\reg}$ extends to a
  finite, étale cover of $\wtilde U$.
\end{thm}

Among its many properties, the covering constructed in
Theorem~\ref{thm:imm1local} can be seen as a simultaneous index-one cover for
all divisors on $X°$ that are $\bQ$-Cartier in a neighbourhood of $p$.  In fact,
the following much stronger result holds true.

\begin{thm}[Simultaneous index-one cover]\label{thm:sidxc}
  In the setting of Theorem~\ref{thm:imm1local}, the following holds for any
  Zariski-open neighbourhood $U = U(p) ⊆ X°$ with preimage $\wtilde U =
  γ^{-1}(U)$.
  \begin{enumerate}
  \item\label{il:mc1l} If $\wtilde D$ is any $\bQ$-Cartier divisor on $\wtilde
    U$, then $\wtilde D$ is Cartier.
  \item\label{il:mc2l} If $D$ is any $\bQ$-Cartier divisor on $U$, then $(\#
    \Gal(γ)) · D$ is Cartier.
  \end{enumerate}
\end{thm}

\begin{rem}[Global bound for the index of $\bQ$-Cartier divisors on klt spaces]
  Under the assumptions of Theorem~\ref{thm:imm1local}, it follows from
  \ref{il:mc2l} and from quasi-compactness of $X$ that there exists a number $N
  ∈ \bN^+$ such that $N · D$ is Cartier, whenever $D$ is a $\bQ$-Cartier
  divisor on $X$.
\end{rem}

\begin{rem}
  It seems to be known to experts that a covering space satisfying a weak
  analogue of \ref{il:mc1l} exists for spaces with rational singularities.  In
  contrast to the usual index-one covers discussed in higher-dimensional
  birational geometry, the Galois group of the morphism $γ$ need not be cyclic.
\end{rem}

As a last direct application of Theorem~\ref{thm:imm1}, we obtain a new proof of
a recent result by Chenyang Xu, \cite[Thm.~2]{Xu12}.  Note that this is not
independent of Xu's work, as the proof of Theorem~\ref{thm:imm1} uses
\cite[Thm.~1]{Xu12}; however, we do not use \cite{HMX14}.

\begin{thm}\label{thm:imm4}
  Let $X$ be a normal, complex, projective variety.  Assume that there exists a
  $\bQ$-Weil divisor $Δ$ such that $(X, Δ)$ is klt, and $-(K_X + Δ)$ is big and
  nef.  Then, the étale fundamental group $\what{π}_1(X_{\reg})$ is finite.
\end{thm}
 In Section~\ref{ssec:GZ} we show by way of example that Theorem~\ref{thm:imm4}
cannot be generalised to rationally connected varieties.

\subsection{Extension for flat sheaves on klt base spaces}

Consider a normal variety $X$ and a flat, locally free, analytic sheaf $\sF°$,
defined on the complex manifold $X_{\reg}^{an}$ associated with the smooth locus
$X_{\reg}$ of $X$.  In this setting, a fundamental theorem of Deligne,
\cite[II.5, Cor.~5.8 and Thm.~5.9]{Deligne70}, asserts that $\sF°$ is algebraic,
and thus extends to a coherent, algebraic sheaf $\sF$ on $X$.  If $X$ is klt, we
will show that Deligne's extended sheaf $\sF$ is again locally free and flat, at
least after passing to a quasi-étale cover.

\begin{thm}[Extension of flat, locally free sheaves]\label{thm:flat}
  Let $X$ be a normal, complex, quasi-projective variety.  Assume that there
  exists a $\bQ$-Weil divisor $Δ$ such that $(X, Δ)$ is klt.  Then, there exists
  a normal variety $\wtilde X$ and a finite, surjective Galois morphism $γ:
  \wtilde X → X$, étale in codimension one, such that the following holds.  If
  $\sG°$ is any flat, locally free, analytic sheaf on the complex space
  $\widetilde X_{\reg}^{an}$, there exists a flat, locally free, algebraic sheaf
  $\sG$ on $\wtilde X$ such that $\sG°$ is isomorphic to the analytification of
  $\sG|_{\wtilde X_{\reg}}$.
\end{thm}

Theorem~\ref{thm:flat} follows as a consequence of Theorem~\ref{thm:imm1}.
Except for the algebraicity assertion, we do not use Deligne's result in our
proof.  In order to avoid confusion, we briefly recall the definition of flat
sheaves.

\begin{defn}[\protect{Flat locally free sheaf}]\label{defn:flat}
  If $Y$ is any complex algebraic variety, and $\sG$ is any locally free,
  analytic sheaf on the underlying complex space $Y^{an}$, we call $\sG$
  \emph{flat} if it is defined by a representation of the topological
  fundamental group $π_1(Y^{an})$.  A locally free, algebraic sheaf on $Y$ is
  called flat if and only if the associated analytic sheaf is flat.
\end{defn}

Using the partial confirmation of the Lipman-Zariski conjecture shown in
\cite{GKKP11}, we obtain the following criterion for a klt space to have
quotient singularities.  We also obtain a first criterion to guarantee that a
given projective variety is a quotient of an Abelian variety.

\begin{cor}[Criterion for quotient singularities and torus quotients]\label{cor:flattangent}
  In the setting of Theorem~\ref{thm:flat}, if $\sT_{X_{\reg}}$ is flat, then
  $\wtilde X$ is smooth and $X$ has only quotient singularities.  If $X$ is
  additionally assumed to be projective, then there exists an Abelian variety
  $A$ and a finite Galois morphism $A → X$ that is étale in codimension
  one.
\end{cor}

\subsection{Characterisation of torus quotients: varieties with vanishing Chern classes}

Consider a Ricci-flat, compact Kähler manifold $X$ whose second Chern class
vanishes.  As a classical consequence of Yau's theorem \cite{MR480350} on the
existence of a Kähler-Einstein metric, $X$ is then covered by a complex torus;
cf.~\cite[Thm.~12.4.3]{MR0278248} and \cite[Ch.~IV, Cor.~4.15]{Kob87}.  Building
on our main result, we generalise this to the singular case, when $X$ has
terminal or a special type of klt singularities.

\begin{thm}[Characterisation of torus quotients]\label{thm:torus}
  Let $X$ be a normal, complex, projective variety of dimension $n$ with at
  worst klt singularities.  Assume that $X$ is smooth in codimension two and
  that the canonical divisor is numerically trivial, $K_X \equiv 0$.  Further,
  assume that there exist ample divisors $H_1, …, H_{n-2}$ on $X$ and a
  desingularisation $π: \widetilde X → X$ such that $c_2(\sT_{\wtilde X}) ·
  π^*(H_1) \cdots π^*(H_{n-2}) = 0$.

  Then, there exists an Abelian variety $A$ and a finite, surjective, Galois
  morphism $A → X$ that is étale in codimension two.
\end{thm}

In fact, the converse is also true; see Theorem~\ref{chern} for a precise
statement.  The three-dimensional case has been settled by Shepherd-Barron and
Wilson in \cite{SBW94}, and our strategy of proof for Theorem~\ref{thm:torus}
partly follows their line of reasoning.  Apart from our main result, the proof of
Theorem~\ref{thm:torus} relies on the semistability of the tangent sheaf of
varieties with vanishing first Chern class, on Simpson's flatness results for
semistable sheaves, \cite[Cor.~3.10]{MR1179076}, and on the partial solution of
the Lipman-Zariski conjecture mentioned above.

\begin{rem}[Vanishing of Chern classes]
  The condition on the intersection numbers posed in Theorem~\ref{thm:torus} is
  a way of saying ``$c_2(X)=0$'' that avoids the technical complications with
  the definition of Chern classes on singular spaces.
  Section~\ref{sec:ChernSing} discusses this in detail.
\end{rem}

\begin{rem}[Terminal varieties, Generalisations]
  Projective varieties with terminal singularities are smooth in codimension
  two, \cite[Cor.~5.18]{KM98}.  Theorem~\ref{thm:torus} therefore applies to
  this class of varieties.  From the point of view of the minimal model program,
  this seems a very natural setting for our problem.

  Using an orbifold version of the second Chern class, Shepherd-Barron and
  Wilson \cite{SBW94} are able to treat threefolds whose singular set has
  codimension two.  It is conceivable that with sufficient technical work, using our line of argumentation a
  similar result could also be obtained in the higher-dimensional setting.  We
  have chosen not to pursue these generalisations here.
\end{rem}

As one important step in the proof of Theorem~\ref{thm:torus}, we generalise
classical flatness results \cite{UhlenbeckYau86, Kob87, MR1179076, MR1463962}
for semistable vector bundles with vanishing first and second Chern classes to
our singular setup, and obtain the following result:

\begin{thm}[Flatness of semistable sheaves with vanishing first and second Chern classes]\label{thm:svcc}
  Let $X$ be an $n$-dimensional, normal, complex, projective variety, smooth in
  codimension two.  Assume that there exists a $\bQ$-Weil divisor $Δ$ such that
  $(X, Δ)$ is klt.  Let $H$ be an ample Cartier divisor on $X$, and $\sE$ be a
  reflexive, $H$-slope-semistable sheaf.  Assume that the following intersection
  numbers vanish
  \begin{equation}\label{eq:svcc}
    c_1(\sE) · H^{n-1} = 0, \quad c_1(\sE)² · H^{n-2} = 0, \quad\text{and}\quad c_2(\sE) · H^{n-2} = 0.
  \end{equation}
  Then, there exists a normal variety $\wtilde X$ and a finite, surjective
  Galois morphism $γ: \wtilde X → X$, étale in codimension one, such that
  $(γ^*\sE)^{**}$ is locally free and flat, that is, $(γ^*\sE)^{**}$ is given by
  a linear representation of $π_1(\wtilde X)$.
\end{thm}

\subsection{Characterisation of torus quotients: varieties admitting polarised endomorphisms}

In \cite{MR2587100}, Nakayama and Zhang study the structure of varieties
admitting polarised endomorphisms ---the notion of ``polarised endomorphism'' is
recalled in Remark~\ref{rem:pE} below.  They conjecture in
\cite[Conj.~1.2]{MR2587100} that any variety of this kind is either uniruled or
covered by an Abelian variety, with a covering map that is étale in codimension
one.  The conjecture has been shown in special cases, for instance in dimensions
less than four.  As an immediate application of Theorem~\ref{thm:mainSpec}, we
show that it holds in full generality.

\begin{thm}[\protect{Varieties with polarised endomorphisms; cf.~\cite[Conj.~1.2]{MR2587100}}]\label{thm:NZ}
  Let $X$ be a normal, complex, projective variety admitting a non-isomorphic
  polarised endomorphism.  Assume that $X$ is not uniruled.  Then, there exists
  an Abelian variety $A$ and a finite, surjective morphism $A → X$ that is étale
  in codimension one.
\end{thm}

\begin{rem}[Polarised endomorphism]\label{rem:pE}
  Let $X$ be a normal, complex, projective variety.  An endomorphism $f: X → X$
  is called \emph{polarised} if there exists an ample Cartier divisor $H$ and a
  positive number $q ∈ \bN^+$ such that $f^* (H) \sim q· H$.
\end{rem}

Theorem~\ref{thm:NZ} has strong implications for the structure of varieties with
endomorphisms.  These are discussed in Section~\vref{ssec:structure}.

\subsection{Outline of the paper}
\approvals{Greb & yes \\ Kebekus & yes \\ Peternell & yes}

Section~\ref{sec:mainTHM} formulates and discusses the main results.  Before
proving these in Part~\ref{part:II}, we have gathered in Part~\ref{part:I} a
number of results and facts which will later be used in the proofs.  While many
of the facts discussed in Sections~\ref{sec:facts} and \ref{sec:bertini} are
known to experts (though hard to find in the literature), the material of
Section~\ref{sec:ChernSing} is new to the best of our knowledge and might be of
independent interest.

The main result and its applications are proven in Part~\ref{part:II}.  The
statements are somewhat delicate and might invite misinterpretation unless care
is taken.  We have therefore chosen to conclude with a number of examples in
Section~\ref{sec:ex}, showing that the assumptions are strictly necessary and
that several ``obvious'' generalisations or reformulations are wrong.
\Preprint{The concluding appendices prove uniqueness and equivariance in
  ``Zariski's Main Theorem in the form of Grothendieck'' and invariance of
  branch loci under Galois closure.}

\Publication{The preprint version of this paper, available as
  \href{http://arxiv.org/abs/1307.5718}{arXiv:1307.5718}, contains additional
  figures, more details and spells out a few standard proofs which are left to
  the reader in the present version.  The numbering is identical.  }

\subsection{Acknowledgements}
\approvals{Greb & yes \\ Kebekus & yes \\ Peternell & ---}

The authors would like to thank Hélène Esnault, Carlo Gasbarri, Patrick Graf, Annette
Huber-Klawitter, James McKernan, Wolfgang Soergel, and Matthias Wendt for
numerous discussions.  Chenyang Xu kindly answered our questions by e-mail.
Angelo Vistoli and Jason Starr responded to Stefan Kebekus' questions on the
MathOverflow web site.  We are grateful to De-Qi Zhang for his interest in our
work and for pointing us towards \cite{AmbroLCtrivial}.  Moreover, we thank
Jochen Heinloth for discussions concerning boundedness of flat sheaves and for
providing references from \cite{MR1320603}.  The authors are particularly
grateful to Fritz Hörmann, who found a mistake in the first preprint version of
this paper.  Finally, the authors would like to thank an anonymous referee for
pointing out several simplifications in the proof of the main result.

%
%
\svnid{$Id: 02.tex 873 2015-07-28 13:14:15Z kebekus $}

\section{Main result}
\label{sec:mainTHM}
\subversionInfo

The introductory section presented a simplified version of our main result.  The
following more general Theorem~\ref{thm:main1} differs from the simplified
version in two important aspects, which we briefly discuss to prepare the reader
for the somewhat technical formulation.  \Publication{The preprint version of
  this paper contains a figure that illustrates the setup
  schematically.}\Preprint{Figure~\vref{fig:main} illustrates the setup
  schematically.
  
  \begin{figure}
    \centering
    \footnotesize
    
    \begin{tikzpicture}
      \draw (0,2) node[right]{\includegraphics[width=4cm]{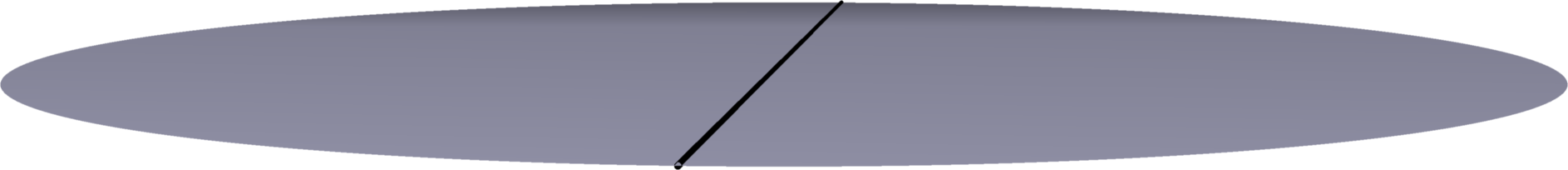}};
      \draw (0.4,2.5) node {$Y_0$};
      \draw [->] (2,1.3) -- node[left]{$η_0$} (2,0.5);
      
      \draw (5.5,2) node[right]{\includegraphics[width=3cm]{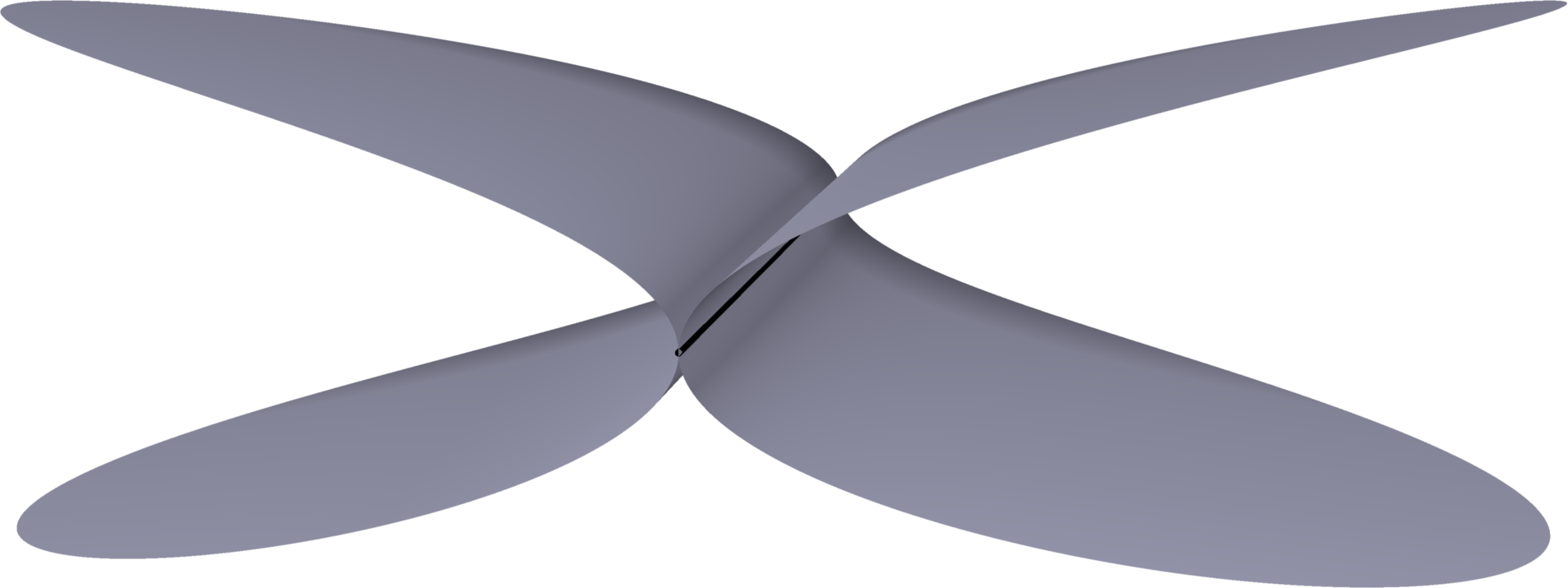}};
      \draw (6.0,2.8) node {$Y_1$};
      \draw [->] (7,1.3) -- node[left]{$η_1$} (7,0.5);
      \draw [->] (5.4,2) -- node[above]{$γ_1$} (4.4,2);
      
      \draw (10,2) node[right]{\includegraphics[width=2cm]{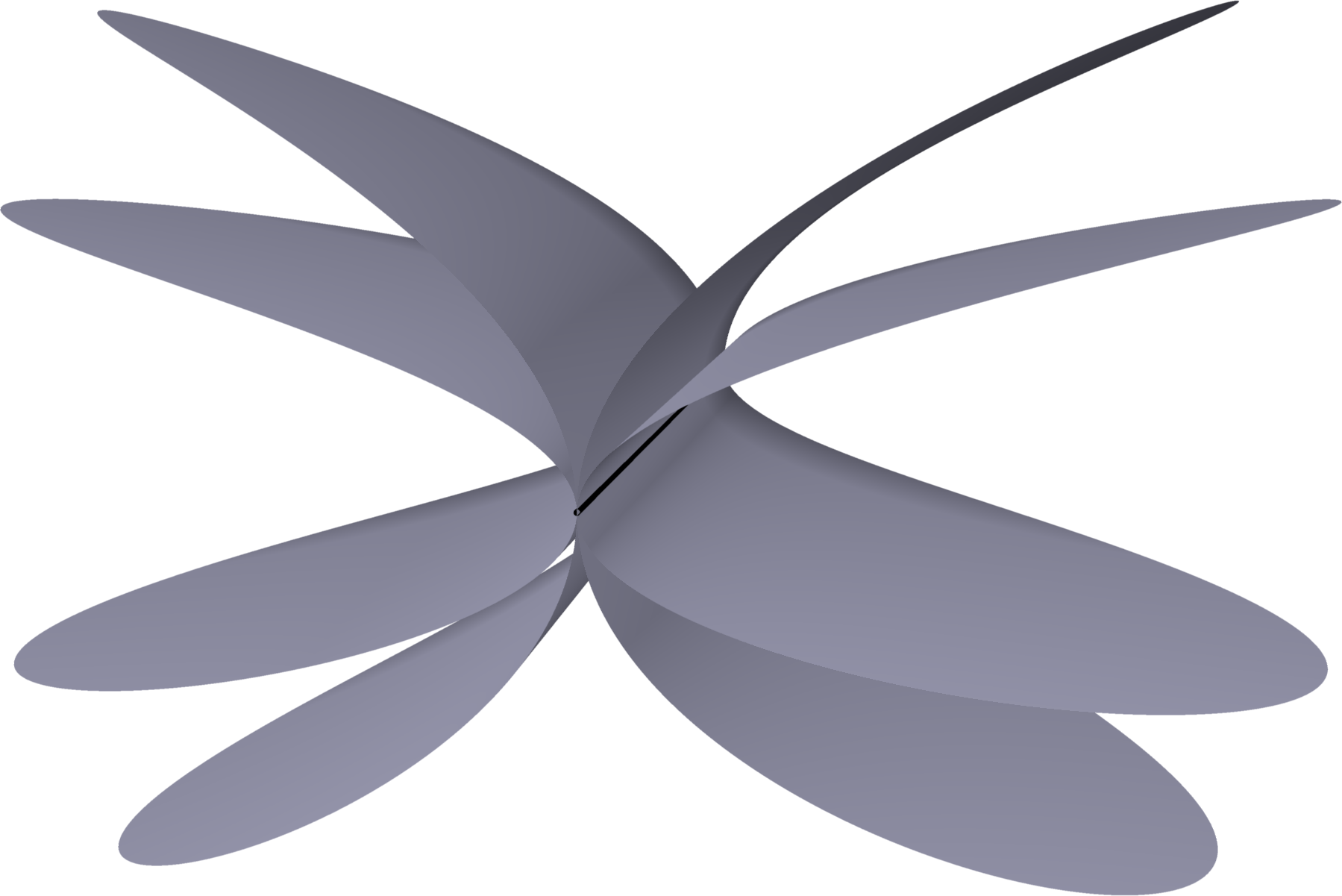}};
      \draw (10.5,2.8) node {$Y_2$};
      \draw [->] (11,1.3) -- node[left]{$η_2$} (11,0.5);
      \draw [->] (9.9,2) -- node[above]{$γ_2$} (8.8,2);
      
      \draw (0,0) node[right]{\includegraphics[width=4cm]{s2-1cover}};
      \draw (0.5,-0.4) node{$X_0$};
      \draw [->] (2.5, -.5) node[right]{$S$} -- (2.0, -.1);
      
      \draw (5.5,0) node[right]{\includegraphics[width=3cm]{s2-1cover}};
      \draw (6.0,-0.4) node{$X_1$};
      \draw [->] (5.4,0) -- node[below]{$ι_1$} (4.4,0);
      
      \draw (10,0) node[right]{\includegraphics[width=2cm]{s2-1cover}};
      \draw (10.5,-0.4) node{$X_2$};
      \draw [->] (9.9,0) -- node[below]{$ι_2$} (8.8,0);
    \end{tikzpicture}
    
    \bigskip
    
    {\small The figure shows the setup for the main result,
      Theorem~\ref{thm:main1}, schematically.  The morphisms $η_i$ are Galois
      covers over a sequence $X ⊇ X_0 ⊇ X_1 ⊇ \cdots$ is increasingly small open
      subsets of $X$.  The morphisms $γ_i$ between these covering spaces are
      étale away from the preimages of $S$.  In Theorem~\ref{thm:main1}, the set
      $S$ is of codimension two or more.  This aspect is difficult to illustrate
      and therefore not properly shown in the figure.}
    
    \caption{Setup for the main result}
    \label{fig:main}
  \end{figure}
}

\subsubsection*{Aspect 1: Finite vs.~quasi-finite morphisms}

For simplicity, we have assumed in Theorem~\ref{thm:mainSpec} that all morphisms
$γ_i$ are finite.  However, there are settings where this assumption is too
restrictive and where one would like to consider quasi-finite rather than finite
morphisms.  The proofs of Theorems~\ref{thm:imm1local} and \ref{thm:sidxc}
provide examples for this more general setup.

To support this kind of application, Theorem~\ref{thm:main1} allows the
morphisms $γ_i$ to be quasi-finite, as long as each $Y_i$ is Galois over a
suitable open subset $X_i ⊆ X$.  For this reason, Theorem~\ref{thm:main1}
introduces a descending chain of dense open subsets $X ⊇ X_0 ⊇ X_1 ⊇ \cdots$ and
replaces Sequence~\eqref{eq:ghh} by the more complicated
Diagram~\eqref{eq:fmapseq-1}.

\subsubsection*{Aspect 2: Specification of the branch locus}

The morphisms $γ_i$ of Theorem~\ref{thm:mainSpec} are required to be étale in
codimension one, and therefore branch only over the singular locus of the target
varieties $Y_{i-i}$.  However, there are settings where it is advantageous to
specify the potential branch locus in a more restrictive manner, requiring that
the $γ_i$ branch only over a given set $S$.  This gives more precise information
on the branch locus.

\begin{thm}[Main result]\label{thm:main1}
  Let $X$ be a normal, complex, quasi-projective variety of dimension $\dim X ≥
  2$.  Assume that there exists a $\bQ$-Weil divisor $Δ$ such that $(X, Δ)$ is
  klt.  Suppose further that we are given a descending chain of dense open
  subsets $X ⊇ X_0 ⊇ X_1 ⊇ \cdots$, a closed reduced subscheme $S ⊂ X$ of
  codimension $\codim_X S ≥ 2$, and a commutative diagram of morphisms between
  normal varieties,
  \begin{equation}\label{eq:fmapseq-1}
    \begin{gathered}
      \xymatrix{
        & Y_0 \ar@{->>}[d]_{η_0} & \ar[l]_{γ_1} Y_1 \ar@{->>}[d]_{η_1} & \ar[l]_{γ_2} Y_2 \ar@{->>}[d]_{η_2} & \ar[l]_{γ_3} Y_3 \ar@{->>}[d]_{η_3} & \ar[l]_{γ_4} \cdots \\
        X & \ar@{_(->}[l]^{ι_0} X_0 & \ar@{_(->}[l]^{ι_1} X_1 & \ar@{_(->}[l]^{ι_2} X_2 & \ar@{_(->}[l]^{ι_3} X_3& \ar@{_(->}[l]^{ι_4} \cdots,
      }
    \end{gathered}
  \end{equation}
  where the following holds for all indices $i ∈ \bN$.
  \begin{enumerate}
  \item\label{il:R1} The morphisms $ι_i$ are the inclusion maps.
  \item\label{il:R2} The morphisms $γ_i$ are quasi-finite, dominant and étale
    away from the reduced preimage set $S_i := η_i^{-1}(S)_{\red}$.
  \item\label{il:R3} The morphisms $η_i$ are finite, surjective, Galois, and
    étale away from $S_i$.
  \end{enumerate}
  Then, all but finitely many of the morphisms $γ_i$ are étale.
\end{thm}

\begin{rem}[Interdependence of conditions in Theorem~\ref{thm:main1}]\label{rem:dep}
  Conditions~\ref{il:R2} and \ref{il:R3} are not independent.  The assumption
  that the morphisms $γ_i$ are étale away from $S_i$ follows trivially from
  Condition~\ref{il:R3} and \cite[Cor.~3.6]{Milne80}.  We have chosen to include
  the redundancy for convenience of reference and notation.
\end{rem}

\part{Preparations}
\label{part:I}

%
%
\svnid{$Id: 03.tex 873 2015-07-28 13:14:15Z kebekus $}

\section{Notation, conventions, and facts used in the proof}
\label{sec:facts}
\subversionInfo

\subsection{Global conventions}
\label{ssec:conventions}

Throughout this paper, all schemes, varieties and morphisms will be defined over
the complex number field.  We follow the notation and conventions of
Hartshorne's book \cite{Ha77}.  In particular, varieties are always assumed to
be irreducible.  For complex spaces and holomorphic maps, we follow \cite{CAS}
whenever possible.  For all notations around Mori theory, such as klt spaces and
klt pairs, we refer the reader to \cite{KM98}.  The empty set has dimension
minus infinity, $\dim \emptyset = - \infty$.

\subsection{Notation}
\approvals{Greb & yes \\ Kebekus & yes \\ Peternell & yes}

In the course of the proofs, we frequently need to switch between the Zariski
and the Euclidean topology.  We will consistently use the following notation.

\begin{notation}[Complex space associated with a variety]
  Given a variety or projective scheme $X$, denote by $X^{an}$ the associated
  complex space, equipped with the Euclidean topology.  If $f : X → Y$ is any
  morphism of varieties or schemes, denote the induced map of complex spaces by
  $f^{an} : X^{an} → Y^{an}$.  If $\sF$ is any coherent sheaf of
  $\sO_X$-modules, denote the associated coherent analytic sheaf of
  $\sO_{X^{an}}$-modules by $\sF^{an}$.
\end{notation}

\begin{defn}[Covers and covering maps, compare with Definition~\ref{def:topcov}]\label{def:cover}
  Let $Y$ be a normal, connected variety or complex space.  A \emph{cover} of
  $Y$ is a surjective, finite morphism (resp.\ holomorphic map) $f : X → Y$,
  where $X$ is again normal and connected.  The map $f$ is called a \emph{covering
    map}.  Two covers $f: X → Y$ and $f': X' → Y$ are called \emph{isomorphic}
  if there exists an isomorphism (resp.\ biholomorphic map) $ψ: X → X'$ such
  that $f' ◦ ψ = f$.  \Preprint{In other words, the following diagram should
    commute:
  $$
  \xymatrix{ %
    X \ar[d]_f \ar[rr]^{ψ}_\cong && \ar[d]^{f'} X' \\
    Y \ar@{=}[rr] && Y.
  }
  $$}
\end{defn}

Note that in Definition~\ref{def:cover} we do not assume $f$ to be étale.
Morphisms that are étale in codimension one will also be called
``quasi-étale''.

\begin{defn}[Quasi-étale morphisms]\label{defn:quasietale}
  A morphism $f : X → Y$ between normal varieties is called \emph{quasi-étale}
  if $f$ is quasi-finite and étale in codimension one.  In other words, $f$ is
  quasi-étale if $\dim X = \dim Y$ and if there exists a closed, subset $Z ⊆ X$
  of codimension $\codim_X Z ≥ 2$ such that
  $f|_{X \setminus Z} : X \setminus Z → Y$ is étale.
\end{defn}

\begin{defn}[\protect{Finite topological covering space; cf.~\cite[Sect.~1.3]{Hatcher}}]\label{def:topcov}
  Given any topological space $Y$, a \emph{finite topological covering space} is
  a topological space $X$ together with a surjective, continuous map $γ : X → Y$
  such that there exists an open cover $\{U_α\}_{α ∈ A}$ of $Y$ with the
  property that for each $α ∈ A$, the preimage $γ^{-1}(U_α)$ is a finite
  disjoint union of open sets in $X$, each of which is mapped homeomorphically
  onto $U_α$ by $γ$.
\end{defn}

\begin{rem}[Comparison with Definition~\ref{def:cover}]
  In contrast to Definition~\ref{def:cover}, we do not assume that covering
  spaces are connected.  A morphism of varieties that is a cover in the sense of
  Definition~\ref{def:cover} is a finite topological covering space if and only
  if it is étale.
\end{rem}

\subsection{Galois morphisms}
\approvals{Greb & yes \\ Kebekus & yes \\ Peternell & yes}

Galois morphisms appear prominently in the literature, but their precise
definition is not consistent.  We will use the following definition.

\begin{defn}[Galois morphism]\label{def:Galois}
  A covering map $γ : X → Y$ of varieties is called \emph{Galois} if there
  exists a finite group $G ⊂ \Aut(X)$ such that $γ$ is isomorphic to the
  quotient map.
\end{defn}

It is a standard, widely-used fact that any finite, surjective morphism between
normal varieties can be enlarged to become Galois.

\begin{thm}[Existence of Galois closure]\label{thm:galoisClosure}
  Let $γ: X → Y$ be a covering map of quasi-projective varieties.  Then, there
  exists a normal, quasi-projective variety $\wtilde X$ and a finite, surjective
  morphism $\wtilde{γ} : \wtilde X → X$ such that the following holds.
  \begin{enumerate}
  \item\label{il:GA1} There exist finite groups $H \leqslant G$ such that the
    morphisms $Γ := γ◦\wtilde{γ}$ and $\wtilde{γ}$ are Galois with group $G$ and
    $H$, respectively.
  \item\label{il:GA2} If $\Br(γ)$ and $\Br(Γ)$ denote the branch loci, with
    their natural structure as reduced subschemes of $Y$, then $\Br(γ) =
    \Br(Γ)$.  \Publication{\qed}
  \end{enumerate}
\end{thm}

The morphism $Γ : \wtilde X → Y$ is often called the \emph{Galois closure of
  $γ$}.  \Preprint{A proof of Theorem~\ref{thm:galoisClosure} is given in
  Appendix~\ref{appendix:B}.}\Publication{A proof of
  Theorem~\ref{thm:galoisClosure} can be found in Appendix~B of the
  \href{http://arxiv.org/abs/1307.5718}{arXiv version of this paper}.}

\subsection{Zariski's Main Theorem}
\approvals{Greb & yes \\ Kebekus & yes \\ Peternell & yes }

Grothendieck observed in \cite{EGA4-3} that Zariski's Main Theorem can be
reformulated by saying that any quasi-finite morphism decomposes as a
composition of an open immersion and a finite map.  If all varieties in question
are normal, this decomposition is unique and well-behaved with respect to the
action of the relative automorphism group.

\begin{thm}[Zariski's Main Theorem in the equivariant setting]\label{thm:ZMT-eq}
  Let $a : V° → W$ be any quasi-finite morphism between quasi-projective
  varieties.  Assume that $V°$ is normal and not the empty set.  Then there
  exists a normal, quasi-projective variety $V$ and a factorisation of $a$ into
  an open immersion $α : V° → V$ and a finite morphism $β : V → W$.  The
  factorisation is unique up to unique isomorphism and satisfies the following
  additional conditions with respect to group actions.
  \begin{enumerate}
  \item\label{il:zmt2} If $G$ is any group that acts on $V°$ and $W$ by
    algebraic morphisms, and if $a$ is equivariant with respect to these
    actions, then $G$ also acts on $V$ and the morphisms $α$ and $β$ are
    equivariant with respect to these actions.
  \item\label{il:zmt3} If $W° := \Image(a)$ is open in $W$ and $a : V° → W°$ is
    Galois with group $G$, then $β: V → W$ is Galois with group
    $G$.  \Publication{\qed}
  \end{enumerate}
\end{thm}

\begin{rem}[Algebraic group actions]
  If the group $G$ in \ref{il:zmt2} is algebraic and acts algebraically on $V°$,
  then the induced action on $V$ will clearly also be algebraic.
\end{rem}

In the affine setting, parts of Theorem~\ref{thm:ZMT-eq} are found in
\cite{MR2210794}.  \Preprint{A full proof of Theorem~\ref{thm:ZMT-eq} is given
  in Appendix~\ref{appendix:A}.}\Publication{The
  \href{http://arxiv.org/abs/1307.5718}{arXiv version of this paper} contains a
  full proof of Theorem~\ref{thm:ZMT-eq} in Appendix~A.}

\subsection{Covers of stratified spaces}
\approvals{Greb & yes \\ Kebekus & yes \\ Peternell & yes}

In this section, we investigate the interplay of branch loci with Whitney
stratifications of the base space.  We follow the terminology of
Goresky-MacPherson's book \cite{GoreskyMacPherson}.  The following proposition
summarises material from \cite[Thm.  in I.1.7 and I.1.4]{GoreskyMacPherson}.

\begin{prop}[Whitney stratification of algebraic varieties and subvarieties]\label{prop:whitneystratexist}
  Let $X$ be an algebraic variety and $A \subsetneq X$ a Zariski-closed subset.
  \begin{enumerate}
  \item\label{il:wsx1} There exists a finite Whitney stratification $(X_i)_{i ∈
      Λ}$ of $X$ into disjoint, locally closed, smooth subvarieties $X_i$ of $X$
    such that $A$ is a union of strata.
  \item\label{il:wsx2} Given any Whitney stratification $(X_i)_{i ∈ Λ}$ as in
    \ref{il:wsx1} and any $i ∈ Λ$, then $X_i$ has a neighbourhood $U_i$ in
    $X^{an}$ that admits a continuous, locally trivial projection $f_i: U_i →
    X_i^{an}$.  Given any $p ∈ X_i$, the fibre $f_i^{-1}(p)$ is homeomorphic to
    the cone over the link of $X_i$ and therefore 1-connected.  \qed
  \end{enumerate}
\end{prop}

\begin{rem}
  Given any Whitney stratification $(X_i)_{i ∈ Λ}$ as in \ref{il:wsx1} and any
  Zariski-open $V ⊆ X$, observe that $(X_i∩V)_{i ∈ Λ}$ is again a Whitney
  stratification of $V$.  In the setting of \ref{il:wsx2}, observe that the
  identity $X_i^{an} → X_i^{an} ⊆ U_i$ yields a section of $f_i$.  The
  restriction $f_i|_{U_i \setminus X_i^{an}}$ is likewise locally trivial.
\end{rem}

\begin{cor}\label{cor:WstratBrach}
  In the setting of Proposition~\ref{prop:whitneystratexist}, assume that
  $\codim_X A ≥ 2$ and let $γ: Y → X$ be a covering map whose branch locus is
  contained in $A$.  If $i ∈ Λ$ is any index such that the Zariski-closure
  $\overline{X_i}$ equals a component of $A$, then either the branch locus
  contains $X_i$ or it is disjoint from $X_i$.
\end{cor}
\begin{proof}
  Given $a ∈ X_i \setminus \Br(γ)$, we need to show that $X_i ∩ \Br(γ) =
  \emptyset$.  Removing strata, we are free to assume that $X_i = A$.  Write
  $$
  U°_i := U_i \setminus X_i, \quad Y° := (γ^{an})^{-1}(U_i°), \quad F := f_i^{-1}(a), \quad F°
  := F ∩ U_i°.
  $$
  Observe that $U°_i$ and $F$ do not intersect the the branch locus.  Consider the long
  exact homotopy sequences,
  \begin{equation}\label{eq:homotopysequences}
    \begin{gathered}
      \begin{xymatrix}
        {
          π_2(X_i) \ar@{=}[d]\ar[r]& π_1(F°) \ar[d]\ar[r] & π_1(U°_i) \ar[d]\ar[r] & π_1(X_i) \ar@{=}[d]\ar[r]& \{1\}\\
          π_2(X_i) \ar[r]& π_1(F)\ar[r] & π_1(U_i) \ar[r]& π_1(X_i)\ar[r] & \{1\},
        }
      \end{xymatrix}
    \end{gathered}
  \end{equation}
  where the vertical arrows are induced by the inclusions.

  Choose $q ∈ F°$.  The finite topological covering space $Y° →
  U°_i$ corresponds to an action of $π_1(U°_i)$ on the finite set $Q :=
  γ^{-1}(q)$, \cite[p.~68ff]{Hatcher}.  Since $γ$ is unbranched over the
  1-connected set $F$, observe that the induced action of $π_1(F°)$ is trivial.
  The action of $π_1(U°_i)$ is therefore induced by an action of $π_1(X_i)$,
  hence by an action of $π_1(U_i)$.  It follows that $γ^{an}|_{Y°}$ can be
  extended to a finite topological covering space $γ': Y' → X$, which carries
  the structure of a normal complex space, locally biholomorphic to $X$.  But
  the two covering spaces $Y^{an}$ and $Y'$ agree outside of $A$, which is of
  codimension two, and are therefore isomorphic by
  \cite[Thm.~3.4]{DethloffGrauert}.  In particular, $γ$ is locally biholomorphic.
\end{proof}

\Preprint{
\subsection{Classification of holomorphic covering maps}
\approvals{Greb & yes \\ Kebekus & yes \\ Peternell & yes}

It is a standard result of topology that covering spaces of a given topological
space are classified by conjugacy classes of the fundamental group.  Using
fundamental results of Grauert, Remmert, and Stein this statement can be
generalised to branched holomorphic coverings of a given normal complex space.

\begin{prop}[Classification of covering maps with prescribed branch locus]\label{prop:SteinTop}
  Let $X$ be a connected, normal complex space and $A \subsetneq X$ a proper
  analytic subset.  Then the following natural map is bijective:
  $$
  \begin{array}{ccc}
    \begin{Bmatrix}
      \text{Isomorphism classes of }\\
      \text{holomorphic covering maps} \\
      \text{$f : Y → X$, where $f$ is local-}\\
      \text{ly biholom.\ away from $f^{-1}(A)$}
    \end{Bmatrix}
    & → &
    \begin{Bmatrix}
      \text{Finite-index subgroups of}\\
      \text{$\what{π}_1(X \setminus A)$ up to conjugation}\\
    \end{Bmatrix} \\[0.6cm]
    f & \mapsto & \left(f|_{Y \setminus f^{-1}(A)}\right)_* \what{π}_1 \bigl(Y \setminus f^{-1}(A) \bigl).
  \end{array}
  $$
\end{prop}

\begin{rem}[Equivalence of categories]
  The map of Proposition~\ref{prop:SteinTop} induces an equivalence of
  categories, but we will not need this stronger statement.
\end{rem}

\begin{proof}[Proof of Proposition~\ref{prop:SteinTop}]
  The classification of finite connected topological covering spaces of $X
  \setminus A$, \cite[Thm.~1.38]{Hatcher}, and the fact that any such cover can
  be given a uniquely determined complex structure that makes the covering map
  holomorphic and locally biholomorphic, \cite[§~1.3]{DethloffGrauert}, together
  establish a correspondence between isomorphism classes of locally
  biholomorphic covering maps $f°: Y° → X\setminus A$ and conjugacy classes of
  finite-index subgroups of $\what{π}_1(X \setminus A)$, by sending $f°$ to
  $(f°)_* \; \what{π}_1 \bigl(Y° \bigl)$.

  In order to prove our claim, it therefore suffices to show that any locally
  biholomorphic holomorphic covering map $f°: Y° → X\setminus A$ can be uniquely
  extended to a holomorphic covering map $f : Y → X$, where $Y$ is a connected
  normal complex space.  This however is a special case of
  \cite[Thm.~3.4]{DethloffGrauert}; for this, observe that the assumption made in
  \cite[Thm.~3.4]{DethloffGrauert} on the analyticity of $A ∪ B$ is
  automatically fulfilled in our setup, since the critical locus of the covering
  that we aim to extend is empty.\end{proof}

\begin{rem}
  Theorem \cite[Thm.~3.4]{DethloffGrauert} is a modern version of results
  obtained by Grauert-Remmert and Stein in their fundamental papers
  \cite{MR0103285, MR0083045} dating back to the 1950's.
\end{rem}

\begin{cor}[Characterisation of biholomorphic maps]\label{cor:SteinTop}
  Consider a sequence of holomorphic covering maps, $W_2 \xrightarrow{γ} W_1
  \xrightarrow{η} V$.  Let $A \subsetneq V$ be any a proper analytic subset such
  that $η$ and $η ◦ γ$ are locally biholomorphic away from $A$.  Then $γ$ is
  biholomorphic if and only if the following two subgroups of $\what{π}_1 \bigl(
  V \setminus A \bigr)$ agree up to conjugation,
  $$
  η_* \; \what{π}_1 \bigl(W_1 \setminus η^{-1}(A) \bigr) \quad \text{and} \quad
  (η ◦ γ)_* \; \what{π}_1 \bigl(W_2 \setminus (η◦ γ)^{-1}(A) \bigr).
  $$
\end{cor}
\begin{proof}
  If $γ$ is biholomorphic, then $W_1 \setminus η^{-1}(A)$ and $W_2 \setminus (η◦
  γ)^{-1}(A)$ are isomorphic over $V \setminus A$.  This implies that the two
  subgroups $η_* \; \what{π}_1 \bigl(W_1 \setminus η^{-1}(A) \bigr)$ and $(η ◦
  γ)_* \; \what{π}_1 \bigl(W_2 \setminus (η◦ γ)^{-1}(A) \bigr)$ agree up to
  conjugation.

  Now assume that $η_* \; \what{π}_1 \bigl(W_1 \setminus η^{-1}(A) \bigr)$ and
  $(η ◦ γ)_* \; \what{π}_1 \bigl(W_2 \setminus (η◦ γ)^{-1}(A) \bigr)$ are
  conjugate.  Then, Proposition~\ref{prop:SteinTop} implies that the covering
  maps $η$ and $η ◦ γ$ are isomorphic.  In other words, there exists a
  biholomorphic map $μ: W_1 → W_2$ such that $(η ◦ γ) ◦μ = η$.  For any point $p
  ∈ V \setminus A$, we therefore have $|(η ◦ γ)^{-1}(p)|= |η^{-1}(p)|$.  Hence,
  the degree of $γ$ is equal to one.  As $W_1$ is normal, the analytic version
  of Zariski's Main Theorem implies that $γ$ is biholomorphic, as claimed.
\end{proof}
  }

%
%
\svnid{$Id: 04.tex 874 2015-07-28 13:16:41Z kebekus $}

\section{Chern classes on singular varieties}
\label{sec:ChernSing}
\subversionInfo

\subsection{Intersection numbers on singular varieties}
\approvals{Greb & yes \\ Kebekus & yes \\ Peternell & yes}

To prove the characterisation of torus quotients given in
Theorem~\ref{thm:torus}, we need to discuss intersection numbers of line bundles
with Chern classes of reflexive sheaves on varieties with canonical
singularities.  For our purposes, it is actually not necessary to define Chern
classes themselves.  The literature discusses several competing notions of Chern
classes on singular spaces, all of which are technically challenging;
cf.~\cite{MacPherson74, Aluffi06}.  For the reader's convenience, we have chosen
to include a short, self-contained presentation, restricting ourselves to the
minimal material required for the proof.

\begin{defn}[Resolution of a space and a coherent sheaf]\label{def:resOfCS}
  Let $X$ be a normal variety and $\sE$ a coherent sheaf of $\sO_X$-modules.  A
  \emph{resolution of $(X,\sE)$} is a proper, birational and surjective morphism
  $π : \wtilde X → X$ such that the space $\wtilde X$ is smooth, and such that
  the sheaf $π^*(\sE)/\tor$ is locally free.  If $π$ is isomorphic over the open
  set where $X$ is smooth and $\sE/\tor$ is locally free, we call $π$ a
  \emph{strong resolution of $(X,\sE)$}.
\end{defn}

\begin{rem}
  In the setup of Definition~\ref{def:resOfCS}, the existence of a resolution of
  singularities combined with a classical result of Rossi,
  \cite[Thm.~3.5]{Rossi68}, shows that resolutions and strong resolutions of
  $(X,\sE)$ exist.
\end{rem}

\begin{defn}[Intersection of Chern class with Cartier divisors]\label{def:A}
  Let $X$ be a normal, $n$-dimensional, quasi-projective variety and $\sE$ be a
  coherent sheaf of $\sO_X$-modules.  Assume we are given a number $i ∈ \bN^+$
  such that $X$ is smooth in codimension $i$ and such that $\sE$ is locally free
  in codimension $i$.  Given any resolution morphism $π: \wtilde X → X$ of $(X,
  \sE)$ and any set of Cartier divisors $L_1, … L_{n-i}$ on $X$, we use the
  following shorthand notation
  $$
  c_i(\sE) · L_1 \cdots L_{n-i} := c_i(\sF) · (π^* L_1) \cdots (π^*
  L_{n-i}) ∈ \bZ.
  $$
  where $\sF := π^*\sE/\tor$, and where $c_i(\sF)$ denotes the classical Chern
  class of the locally free sheaf $\sF$ on the smooth variety $\wtilde X$.  More
  generally, if $P(y_1, …, y_n)$ is a homogeneous polynomial of degree $i$ for
  weighted variables $y_j$ with $\deg y_j = j$, we will use the shorthand
  notation
  $$
  P\bigl(c_1(\sE), …, c_n(\sE)\bigr) · L_1 \cdots L_{n-i} := P
  \bigl(c_1(\sF), …, c_n(\sF) \bigr) · (π^* L_1) \cdots (π^* L_{n-i}).
  $$
\end{defn}

\begin{rem}[Independence of resolution]
  \Preprint{In the setting of Definition~\ref{def:A}, given another smooth
    variety $\what X$ and a diagram
    $$
    \xymatrix{ %
      \what X \ar[rrrr]_{η, \text{ proper, surjective, birational}}
      \ar@/^0.5cm/[rrrrrr]^{γ := π ◦ η}&&&& \wtilde X \ar[rr]_{π}
      && X, }
    $$
    then $(γ^* \sE)/\tor \cong η^* \bigl( (π^* \sE)/\tor \bigr)$.  It follows
    immediately from the projection formula for Chern classes, \cite[Thm.~3.2 on
    p.~50]{Fulton98}, that
    $$
    c_i\bigl( (π^*\sE)/\tor) · (π^* L_1) \cdots (π^* L_{n-i}) =
    c_i\bigl( (γ^*\sE)/\tor \bigr) · (γ^* L_1) \cdots (γ^*
    L_{n-i}).
    $$
    The same holds for the more complicated expressions involving the polynomial
    $P$.} Since any two resolution maps are dominated by a common third, the
  numbers defined in Definition~\ref{def:A} are independent of the choice of the
  resolution map.
\end{rem}

\begin{rem}[Alternative computation for semiample divisors]\label{rem:altComp}
  In the setup of Definition~\ref{def:A}, assume we are given numbers
  $m_1, …, m_{n-i} ∈ \bN^+$ and basepoint-free linear systems $B_j ⊆ |m_j L_j|$,
  for all $1 ≤ j ≤ n-i$.  Choose general elements $Δ_j ∈ B_j$ and consider the
  intersection
  $$
  S := Δ_1 ∩ \cdots ∩ Δ_{n-i}.
  $$
  Observe that $S$ is smooth, entirely contained in the smooth locus $X_{\reg} ⊆
  X$, that $\sE$ is locally free along $S$, and that
  $$
  P\bigl(c_1(\sE),…, c_n(\sE)\bigr) · L_1 \cdots L_{n-i} = \frac{P\bigl( c_1
    \bigl(\sE|_S\bigr),…, c_n\bigl(\sE|_S\bigr) \bigr)}{m_1 \cdots m_{n-i}} ∈
  \bZ,
  $$
  where the left hand side is given as in Definition~\ref{def:A} above.
\end{rem}

The following proposition is a fairly direct consequence of
Remark~\ref{rem:altComp} and the projection formula.

\begin{prop}[Behaviour under covers, étale in codimension $i$]\label{prop:3112}
  In the setup of Definition~\ref{def:A}, if $γ : \wtilde X → X$ is any cover,
  étale in codimension $i$, then $\wtilde X$ is smooth in codimension $i$, the
  sheaf $γ^*\sE$ is locally free in codimension $i$, and \Publication{
    \begin{multline*}
      (\deg γ) · P\bigl(c_1(\sE), …, c_n(\sE)\bigr) · L_1 \cdots L_{n-i} \\
      = P\bigl(c_1(γ^*\sE), …, c_n(γ^*\sE)\bigr) · (γ^*L_1) \cdots (γ^*L_{n-i}).  \qed
    \end{multline*}
  }
  \Preprint{
    \begin{multline*}
      (\deg γ) · P\bigl(c_1(\sE), …, c_n(\sE)\bigr) · L_1 \cdots L_{n-i} \\
      = P\bigl(c_1(γ^*\sE), …, c_n(γ^*\sE)\bigr) · (γ^*L_1) \cdots (γ^*L_{n-i}).
    \end{multline*}
  }
\end{prop}
\Preprint{\begin{proof}
    The statements about smoothness of $\wtilde X$ and about local freeness of
    $γ^*\sE$ follow from purity of the branch locus.
    
    Since any Cartier divisor is linearly equivalent to the difference of two
    very ample divisors, we can assume without loss of generality that the
    divisors $L_1, …, L_{n-i}$ are very ample and general in basepoint-free
    linear systems.  The intersection $S := L_1 ∩ \cdots ∩ L_{n-i}$ is then
    smooth, entirely contained in $X_{\reg}$, and the morphism $γ$ is étale near
    $S$.  Writing $\wtilde S := γ^{-1}(S) = γ^* L_1 ∩ \cdots ∩ γ^*L_{n-i}$,
    observe that the alternative computation of Remark~\ref{rem:altComp} applies
    to both $S$ and $\wtilde S$, and yields the following.
    \begin{multline*}
      P\bigl(c_1(γ^*\sE), …, c_n(γ^*\sE)\bigr) · (γ^*L_1) \cdots (γ^*L_{n-i}) \\
      \begin{aligned}
      & = P\bigl( c_1 \bigl( γ^*\sE|_{\wtilde S} \bigr), \cdots, c_n \bigl( γ^*\sE|_{\wtilde S} \bigr) \bigr) && \text{Remark~\ref{rem:altComp} for $\wtilde X$} \\
      & = \deg γ · P\bigl( c_1( \sE|_S ), …, c_n( \sE|_S ) \bigr)&& \text{Projection Formula} \\
      & = \deg γ · P\bigl( c_1(\sE), …, c_1(\sE) \bigr) · L_1 \cdots L_{n-i} && \text{Remark~\ref{rem:altComp} for $X$.}
      \end{aligned}
    \end{multline*}
    This finishes the proof of Proposition~\ref{prop:3112}.
\end{proof}}

\subsection{Numerically trivial Chern classes}

Theorem~\ref{thm:torus} discusses varieties $X$ where the intersection numbers
$c_2(\sT_X) · H_1 \cdots H_{n-2}$ vanish for certain ample divisors $H_1,
\dots, H_{n-2}$.  We will see in Proposition~\ref{prop:trivialityCriterion} that
this sometimes implies that all intersection numbers with arbitrary divisors
vanish.  The following definition is relevant in the discussion.

\begin{defn}[Numerically trivial Chern class]\label{def:B}
  In the setting of Definition~\ref{def:A}, we say that \emph{$c_i(\sE)$ is
    numerically trivial}, if $c_i(\sE) · L_1 \cdots L_{n-i} = 0$ for all
  Cartier divisors $L_1, …, L_{n-i}$ on $X$.
\end{defn}

\begin{prop}[Criterion for numerical triviality of $c_2(\sT_X)$]\label{prop:trivialityCriterion}
  Let $X$ be a normal, $n$-dimensional, projective variety that is smooth in
  codimension two.  Assume that the canonical divisor $K_X$ is $\bQ$-Cartier and
  nef.  Then, $c_2(\sT_X)$ is numerically trivial if and only if there
  \emph{exist} ample Cartier divisors $H_1, …, H_{n-2}$ on $X$ such that
  \begin{equation}\label{eq:mxtriv}
    c_2(\sT_X) · H_1 \cdots H_{n-2} = 0
  \end{equation}
\end{prop}

Before proving Proposition~\ref{prop:trivialityCriterion} in
Section~\ref{ssec:PFtricrit} below, we note the following immediate corollary of
Propositions~\ref{prop:3112} and \ref{prop:trivialityCriterion}.

\begin{cor}[Behaviour of numerically trivial $c_2(\sT_X)$ under coverings]\label{cor:311}
  Let $X$ be a normal, $n$-dimensional, projective variety that is smooth in
  codimension two.  Assume that the canonical divisor $K_X$ is $\bQ$-Cartier and
  nef.  If $γ : X' → X$ is any cover, étale in codimension two, then $X'$ is smooth in
  codimension two, and $K_{X'}$ is $\bQ$-Cartier and nef.  Moreover,
  $c_2(\sT_X)$ is numerically trivial if and only if $c_2(\sT_{X'})$ is
  numerically trivial.  \qed
\end{cor}

\subsection{Proof of Proposition~\ref*{prop:trivialityCriterion}}
\label{ssec:PFtricrit}

The proof of Proposition~\ref{prop:trivialityCriterion} spans the current
Section~\ref{ssec:PFtricrit}.  For the convenience of the reader, we have
subdivided the proof into several relatively independent steps.

\subsection*{Step 1: Setup of notation}
\label{sssec:SX1}

Let $H_1, …, H_{n-2}$ be ample Cartier divisors on $X$ such that
\eqref{eq:mxtriv} holds.  Let further Cartier divisors $L_1, …, L_{n-2}$ be
given, and let $π: \wtilde X → X$ be a strong resolution of singularities.  To
prove Proposition~\ref{prop:trivialityCriterion}, we need to show that
\begin{equation}\label{eq:mxtriv2}
  c_2(\sT_X) · L_1 \cdots L_{n-2} = 0.
\end{equation}
Since each of the $L_i$ can be written as a difference of very ample divisors,
it suffices to show the vanishing under the additional assumption that each of
the $L_i$ is ample.  Replacing the $H_i$ with sufficiently high powers, we can
even assume without loss of generality that the following holds.

\begin{asswlog}\label{ass:amplitude}
  For any index $i$, the divisors $L_i$, $H_i + L_i$ and $H_i - L_i$ are ample.
\end{asswlog}

\subsection*{Step 2: Miyaoka semipositivity}

In order to prove \eqref{eq:mxtriv2} we shall make use of Miyaoka's Semipositivity
Theorem.  In our setup, it asserts that the intersection numbers of $c_2(\sT_X)$
with nef Cartier divisors on $X$ are never negative.

\begin{claim}[Miyaoka semipositivity]\label{claim:MiyaSP}
  Assumptions as above.  If $A_1, …, A_{n-2}$ are ample Cartier divisors
  on $X$, then $c_2(\sT_X) · A_1 \cdots A_{n-2} ≥ 0$.
\end{claim}
\begin{proof}
  Let $S ⊂ X$ be a complete intersection surface, as introduced in
  Remark~\ref{rem:altComp}.  Observe that $S$ avoids the singularities of $X$,
  and that the strong resolution map $π$ is isomorphic along $S$.  Since
  $π^*(\sT_X)$ and $\sT_{\wtilde X}$ differ only along the $π$-exceptional set,
  Remark~\ref{rem:altComp} shows that
  $$
  c_2(\sT_X) · A_1 \cdots A_{n-2} = c_2(\sT_{\wtilde X}) · (π^* A_1)
  \cdots (π^* A_{n-2})
  $$
  Using the assumption that the canonical divisor $K_X$ is $\bQ$-Cartier and
  nef, Miyaoka shows in \cite[Thm.~6.6]{Miyaoka87} that the right hand side is
  always non-negative.
\end{proof}

As a first corollary of Claim~\ref{claim:MiyaSP} we obtain the following technical
result, which shows vanishing of a larger class of intersection numbers.

\begin{consequence}\label{cons:intA}
  Assumptions as above.  For all numbers $1≤ k≤ n-2$, we have the
  following vanishing of numbers,
  \begin{equation}\label{eq:mxtriv3}
    c_2(\sT_X) · (H_1 + L_1) \cdots (H_k + L_k) · H_{k+1} \cdots
    H_{n-2} = 0.
  \end{equation}
\end{consequence}
\begin{proof}
  We prove Consequence~\ref{cons:intA} by induction on $k$.  For $k=0$,
  Equation~\eqref{eq:mxtriv3} equals Equation~\eqref{eq:mxtriv}, which holds by
  assumption.  For the inductive step, assume that Equation~\eqref{eq:mxtriv3}
  holds for a given number $k$.  We need to show that it holds for $k+1$.  To
  this end, consider the following two computations,
  \begin{align}
    \label{eq:xxDSa} 0 & ≤ c_2(\sT_X) · (H_1 + L_1) \cdots (H_k + L_k) · (H_{k+1} \pm L_{k+1}) · H_{k+2} \cdots H_{n-2} \\
    \label{eq:xxDSb} & = \pm c_2(\sT_X) · (H_1 + L_1) \cdots (H_k + L_k) · L_{k+1} · H_{k+2} \cdots H_{n-2}.
  \end{align}
  The Inequalities~\eqref{eq:xxDSa} hold by Claim~\ref{claim:MiyaSP}, because
  the two bundles $H_{k+1} \pm L_{k+1}$ are both ample by
  Assumption~\ref{ass:amplitude}.  The Equalities~\eqref{eq:xxDSb} hold by
  induction hypothesis.  Together, the inequalities show that the two numbers
  computed must both be zero.  This finishes the proof of
  Consequence~\ref{cons:intA}.
\end{proof}

\subsection*{Step 3: End of proof}
\label{sssec:SX3}

To prove Equation~\eqref{eq:mxtriv2} and thereby finish the proof of
Proposition~\ref{prop:trivialityCriterion}, apply Consequence~\ref{cons:intA}
for $k=n-2$.  We obtain the equation
\begin{equation}\label{eq:intA2}
  c_2(\sT_X) · (H_1 + L_1) \cdots (H_{n-2} + L_{n-2}) = 0.
\end{equation}
Multiplying out, we write this number as a sum of the form
$$
0 = \sum\nolimits_j c_2(\sT_X) · A_{1,j} \cdots A_{n-2,j} \quad\text{where }
A_{i,j} ∈ \{H_i, L_i\} \text{ for all indices }i.
$$
Recalling that the divisors $H_i$ and $L_i$ are ample, Claim~\ref{claim:MiyaSP}
thus asserts that none of the summands is negative.  Summing up to zero, we
obtain that all summands must actually be zero.  To finish the proof of
Proposition~\ref{prop:trivialityCriterion}, observe that $c_2(\sT_X) · L_1
\cdots L_{n-2}$ is one of the summands involved.  \qed

%
%
\svnid{$Id: 05.tex 874 2015-07-28 13:16:41Z kebekus $}

\section{Bertini-type theorems for sheaves and their moduli}
\approvals{Greb & yes \\ Kebekus & yes \\ Peternell & yes}
\subversionInfo
\label{sec:bertini}

The proofs of our main results use the fact that two reflexive sheaves belonging
to a bounded family are isomorphic if and only if their restrictions to a
general hyperplane of sufficiently high degree agree.  The present section is
devoted to the proof of this auxiliary result.

\begin{prop}[Testing isom.~classes of reflexive sheaves on hyperplanes]\label{prop:rflxBertini-H}
  Let $X$ be a normal, projective variety of dimension $\dim X ≥ 2$ and $\sE$,
  $\sF$ two coherent, reflexive sheaves of $\sO_X$-modules.  Let $H ⊂ X$ be an
  irreducible, reduced, ample Cartier divisor on $X$ such that the following
  holds.
  \begin{enumerate}
  \item\label{il:pq1} The variety $H$ is normal.
  \item\label{il:pq2} Setting $\sH := \sHom \bigl(\sE,\, \sF\bigr)$, the
    restricted sheaves $\sE|_H$, $\sF|_H$ and $\sH|_H$ are reflexive.
  \item\label{il:pq3} Let $X' \subsetneq X$ be the minimal closed set such that
    $\sE$ and $\sF$ are locally free outside of $X'$.  Setting $H' := X' ∩ H$,
    we have $\codim_H H' ≥ 2$.
  \item\label{il:pq4} The cohomology group $H¹ \bigl( X, \, \sI_H ⊗ \sH
    \bigr)$ vanishes.
  \end{enumerate}
  Then, $\sE$ is isomorphic to $\sF$ if and only if the sheaves $\sE|_H$ and
  $\sF|_H$ are isomorphic.  \Publication{\qed}
\end{prop}

\Preprint{A proof is given below.}\Publication{A full proof is given in the
  \href{http://arxiv.org/abs/1307.5718}{arXiv version of this paper}.} As a
consequence, we obtain the following result.

\begin{prop}[Bertini-type theorem for isom.~classes of reflexive sheaves]\label{prop:rflxBertini}
  Let $X$ be a normal, projective variety of dimension $\dim X ≥ 2$ and $\sE$,
  $\sF$ two coherent, reflexive sheaves of $\sO_X$-modules.  If $\sL ∈ \Pic(X)$
  is ample, then there exists a number $M ∈ \bN$ and for any $m ≥ M$ a dense
  open subset $V_m ⊆ |\sL^{⊗ m}|$ such that all hyperplanes $H ∈ V_m$ are
  reduced and normal, and such that $\sE \cong \sF$ if and only if $\sE|_H \cong
  \sF|_H$.
\end{prop}
\begin{proof}
  Consider the sheaf $\sH := \sHom \bigl(\sE,\, \sF\bigr)$.  Since $\sF$ is
  assumed to be reflexive, so is $\sH$.  Since $X$ is normal, $\sH$ hence has
  depth $≥ 2$ at every point of $X$, \cite[Prop.~1.3]{MR597077}.  It hence
  follows from \cite[Exp.~XII, Prop.~1.5]{SGA2} that
  \begin{equation}\label{eq:ltrz}
    H¹ \bigl( X, \, (\sL^*)^{⊗ m} ⊗ \sH \bigr) = 0 \quad \text{for all }m \gg 0.
  \end{equation}
  Choose one $M ∈ \bN$ such that Equation~\eqref{eq:ltrz} holds for all $m ≥ M$,
  and such that the linear systems $|\sL^{⊗ m}|$ are basepoint-free.
  
  To prepare for the construction of $V_m$, recall from
  \cite[Cor.~1.4]{MR597077} that $\sE$, $\sF$ and $\sH$ are locally free outside
  of a closed subset $X' \subsetneq X$ with $\codim_X X' ≥ 2$.  Now, given any
  number $m ≥ M$, let $V_m ⊆ |\sL^{⊗ m}|$ be the maximal open set such
  that the following holds for all hyperplanes $H ∈ V_m$.
  \begin{enumerate}
  \item The hyperplane $H$ is irreducible, reduced and normal.
  \item\label{il:J2} The intersection $H' := H ∩ X'$ is small, that is,
    $\codim_H H' ≥ 2$.
  \item\label{il:J3} The restricted sheaves $\sE|_H$, $\sF|_H$ and $\sH|_H$ are
    reflexive.
  \end{enumerate}
  Seidenberg's theorem, \cite[Thm.~1.7.1]{BS95}, Bertini's theorem, and
  \cite[Thm.~12.2.1]{EGA4-3} guarantee that none of the open sets $V_m$ is
  empty.  Together with \eqref{eq:ltrz}, Proposition~\ref{prop:rflxBertini-H}
  therefore applies to all $H ∈ V_m$.
\end{proof}

Let $X$ be a normal projective variety of dimension $\dim X ≥ 2$, let $\sE$ be a
coherent, reflexive sheaf of $\sO_X$-modules, and $F$ be a bounded family of
locally free sheaves.  If $H$ is any sufficiently ample Cartier divisor, then
$$
H¹ \bigl( X,\, \sJ_H ⊗ \sHom (\sE,\, \sF) \bigr) = 0 \quad \text{for all } \sF ∈ F.
$$
Note that $\sF|_H$ is locally free for all $\sF ∈ F$ and for all hyperplanes $H
∈ |\sL^{⊗ m}|$.  With the same arguments as above, an iterative
application of Proposition~\ref{prop:rflxBertini-H} therefore also yields the
following.

\begin{cor}[Iterated Bertini-type theorem for bounded families]\label{cor:bertstabY}
  Let $X$ be a normal, projective variety of dimension $\dim X ≥ 2$.  Let $\sE$
  be a coherent, reflexive sheaf of $\sO_X$-modules, and let $F$ be a bounded
  family of locally free sheaves.  Given an ample line bundle $\sL ∈ \Pic(X)$, a
  sufficiently increasing sequence $0 \ll m_1 \ll m_2 \ll \cdots \ll m_{k}$ and
  general elements $H_i ∈ |\sL^{⊗ m_i}|$ with associated complete
  intersection variety $S := H_1 ∩ \cdots ∩ H_k$, then the following holds for
  all sheaves $\sF ∈ F$.  The sheaf $\sF$ is isomorphic to $\sE$ if and only if
  $\sF|_S$ is isomorphic to $\sE|_S$.  \qed
\end{cor}

\Preprint{
\subsection*{Proof of Proposition~\ref*{prop:rflxBertini-H}}

The implication ``$\sE \cong \sF \Rightarrow \sE|_H \cong \sF|_H$'' is clear.
For the opposite direction assume for the remainder of the proof that we are
given isomorphism $λ_H : \sE|_H → \sF|_H$.  Using the vanishing \ref{il:pq4}, we
aim to extend $λ_H$ to a morphism $λ : \sE → \sF$, which will turn out to be
isomorphic.

\subsubsection*{Step 1: Restriction of $\sH$}
\label{ssec:rH}

We compare the restriction $\sH|_H$ to the homomorphism sheaf associated with
the restrictions, $\sH_H := \sHom \bigl( \sE|_H,\, \sF|_H \bigr)$.  Note that
the latter has a non-trivial section given by $λ_H$.  Item~\ref{il:pq2} implies
that $\sH|_H$ and $\sH_H$ are reflexive sheaves of $\sO_H$-modules.  It is also
clear that outside of the small set $H' ⊂ H$ the sheaves $\sH|_H$ and $\sH_H$
are both locally free and agree.  But since two reflexive sheaves are isomorphic
if and only if they are isomorphic on the complement of a small set, this
immediately yields an isomorphism $\sH|_H \cong \sH_H$.

\subsubsection*{Step 2: Extension of the morphism $λ_H$}

We will now show that there exists a morphism $λ : \sE → \sF$ such that $λ_H =
λ|_H$.  To this end, consider the ideal sheaf sequence of the hypersurface $H$,
that is, $0 → \sJ_H → \sO_X → \sO_H → 0$.  Since $\sH$ is torsion free, this
sequence stays exact when tensoring with $\sH$, \cite[Sect.~3.1]{Weibel94}.  The
associated long exact sequence cohomology sequence then reads
\begin{equation}\label{eq:dflj}
  \cdots → H^0 \bigl( X,\, \sH \bigr) \xrightarrow{r} H^0 \bigl( X,\, \sH|_H \bigr)
  → \underbrace{H¹ \bigl( X,\, \sJ_H ⊗ \sH)}_{= 0\text{ by~\ref{il:pq4}}} → \cdots,
\end{equation}
where $r$ is the natural restriction map.  We have seen in Step~1 that the
middle term in \eqref{eq:dflj} is isomorphic to $H^0 \bigl( X,\, \sH_H \bigr)$,
which contains $λ_H$.  Surjectivity of $r$ therefore implies that $λ_H$ can be
extended.  Choose one extension $λ$ and fix this choice throughout.  To finish
the proof, we need to show that $λ$ is an isomorphism.

\subsubsection*{Step 3: Injectivity of $λ$}

Both kernel and image of the morphism $λ$ are subsheaves of the torsion free
sheaves $\sE$ and $\sF$, respectively, and therefore themselves torsion free.
Let $X° ⊆ X$ be the maximal open subset, where $\ker(λ)$, $\sE$ and $\img(λ)$
are locally free.  Recalling from \cite[Cor.~1.4]{MR597077} that $\codim_X X
\setminus X° ≥ 2$, the intersection $H° := H ∩ X°$ is clearly non-empty, and the
restricted sequence
$$
0 → \ker(λ)|_{H°} → \sE|_{H°} \xrightarrow{λ|_{H°}} \img(λ)|_{H°} → 0
$$
remains exact.  Since $λ|_{H°} = λ_H|_{H°}$ is isomorphic, the sheaf
$\ker(λ)|_{H°}$ vanishes, showing that the morphism $λ$ is injective in an open
neighbourhood of $H°$.  Since $\sE$ is reflexive, hence torsion free, it follows
that $λ$ is injective.

\subsubsection*{Step 4: Surjectivity of $λ$, end of proof}

Restriction to $H$ is a functor that is exact on the right.  It follows that
$\coker(λ)|_H = \coker(λ|_H) = 0$.  In particular, we see that the support of
$\coker(λ)$ does not intersect the hyperplane $H$.  Since $H$ is ample, it
follows that the support of $\coker(λ)$ is finite, and that the injective map
$λ$ is isomorphic away from this finite set.  Using the standard fact that two
reflexive sheaves are isomorphic if and only if they are isomorphic away from a
small set, it follows that $λ$ is an isomorphism.  This finishes the proof of
Proposition~\ref{prop:rflxBertini}.  \qed}

\part{Quasi-étale covers of klt spaces}
\label{part:II}

%
%
\svnid{$Id: 06new.tex 780 2014-05-13 15:00:19Z kebekus $}

\section{Proof of Theorems~\ref*{thm:main1} and \ref*{thm:mainSpec}}
\label{sec:P1}
\subversionInfo

\subsection{Proof of Theorem~\ref*{thm:main1}}\label{subsect:proofofMainTheorem}
\approvals{Greb & yes \\ Kebekus & yes \\ Peternell & yes}

Maintaining notation and assumptions of Theorem~\ref{thm:main1}, write $δ_{j,i}
:= γ_{i+1} ◦ \cdots ◦ γ_j : Y_j → Y_i$.  \Publication{Observe}\Preprint{Recall
  from \cite[II~Cor.\ 4.8.e]{Ha77}} that $δ_{j,i}$ is finite over
$η_i^{-1}(X_j)$ and write
$$
T_i := \overline{\bigcup\nolimits_{j>i} \Br δ_{j,i}} ⊆ S_i \subsetneq Y_i.
$$
We may assume that $T_0 \neq \emptyset$.  The proof proceeds by induction over
$\dim T_0$.

\subsubsection{Start of Induction}
\approvals{Greb & yes \\ Kebekus & yes \\ Peternell & yes}
\CounterStep

Assume that $\dim T_0 = 0$.  Choose a point $t ∈ T_0$.  We may assume that $t ∈
η_0^{-1}(X_i)$ for all $i ∈ \bN$, for otherwise there is nothing to show.  Next,
choose a descending sequence of analytically open neighbourhoods of $t ∈
Y_0^{an}$, say $(V_i)_{i∈ \bN}$, such that the following holds.
\begin{enumerate}
\item The set $V_0$ intersects $T_0$ precisely in $t$.  The set $V_0 \setminus
  \{t\}$ is homeomorphic to the open cone over the link $\Link(Y_0,t)$.

\item For positive indices $i$, there are inclusions $V_i ⊆ η_0^{-1}(X_i)^{an}$.
  The connected components of $(δ_{i,0}^{an})^{-1}(V_i)$ each contain exactly
  one point over $T_0$.

\item The inclusion maps $V_i \setminus \{t\} → V_0 \setminus \{t\}$ are
  homotopy equivalences.
\end{enumerate}
We refer the reader to \cite[Sect.~2.3.2]{CAS}, \cite[Thm.~5.1]{MR1194180} for
the results used here.  Next, choose a sequence of connected components, $W_i ⊆
(δ_{i,0}^{an})^{-1}(V_i)$ such that $γ_i^{an}(W_i) ⊆ W_{i-1}$.  Let $t_i ∈
W_i$ be the unique point lying over $t$.  We obtain a descending sequence of
subgroups,
$$
G_i := (δ_{i,0}^{an})^{\vphantom{an}}_* \: \what{π}_1 \bigl(W_i \setminus
\{t_i\} \bigr) \;⊆\; \what{π}_1 \bigl(V_i \setminus \{t\} \bigr) = \what{π}_1
\bigl(V_0 \setminus \{t\} \bigr).
$$

\begin{rem}\label{rem:R1}
  It follows from \ref{il:R3} that the morphisms $δ_{i,0}$ are Galois.  If $i$
  is any index and $W'_i \ne W_i$ is any other connected component of
  $(δ_{i,0}^{an})^{-1}(V_i)$ with associated group $G'_i$, then $G_i$ and $G'_i$
  are both normal, and in fact equal.  The groups $G_i$ do therefore not depend
  on the specific choice of the $W_i$.
\end{rem}

The descending sequence $(G_i)_{i ∈ \bN^+}$ stabilises because the algebraic
local fundamental group $\what{π}_1^{loc}(X,s) := \what{π}_1 \bigl (\Link(Y_0,
t) \bigr)$ of the klt base space $Y_0$ is finite, \cite[Thm.~1]{Xu12}.  This
allows us to choose $M ∈ \bN$ such that $G_i = G_M$ for all $i > M$.  It follows
\Preprint{from Corollary~\ref{cor:SteinTop}} that the morphisms $γ_i$ are étale
over $t$, for all $i > M$.  This finishes the proof in case where $T_0$ is
finite.

\subsubsection{Inductive step}
\approvals{Greb & yes \\ Kebekus & yes \\ Peternell & yes}

Assume that $\dim T_0 > 0$ and that Theorem~\ref{thm:main1} has already been
shown for all diagrams with smaller-dimensional $T_0$'s.  Choose Whitney
stratifications of the varieties $Y_i$ as in
Proposition~\ref{prop:whitneystratexist}, such that the subsets $T_i \subsetneq
Y_i$ are unions of strata.  If $A \subsetneq X$ is a very general hyperplane,
the following will hold.
\begin{enumerate}
\item The hyperplane $A$ is irreducible, normal, not contained in $\supp Δ$, and
  $\bigl( A, Δ|_A \bigr)$ is klt.
\item The set $S_A := S ∩ A$ has codimension at least two in $A$.
\item For every $i ∈ \bN$, the hyperplane $B_i := η_i^{-1}(A)$ intersects every
  positive-dimensional stratum of $Y_i$ non-trivially with the appropriate
  codimension.  In particular, $B_i \not ⊂ T_i$ and $\dim B_0 ∩ T_0 <
  \dim T_0$.
\end{enumerate}
Write $A_k := A ∩ X_k$, note that is non-empty for all $k ≥ 0$, and observe that
the restriction of Diagram~\eqref{eq:fmapseq-1},
$$
\xymatrix{
  & B_0 \ar@{->>}[d]_{η_0|_{B_0}} & \ar[l]_{γ_1|_{B_1}} B_1 \ar@{->>}[d]_{η_1|_{B_1}} & \ar[l]_{γ_2|_{B_2}} B_2 \ar@{->>}[d]_{η_2|_{B_2}} & \ar[l]_{γ_3|_{B_3}} B_3 \ar@{->>}[d]_{η_3|_{B_3}} & \ar[l]_{γ_4|_{B_4}} \cdots \\
  A & \ar@{_(->}[l]^{ι_0|_{A_0}} A_0 & \ar@{_(->}[l]^{ι_1|_{A_1}} A_1 & \ar@{_(->}[l]^{ι_2|_{A_2}} A_2 & \ar@{_(->}[l]^{ι_3|_{A_3}} A_3& \ar@{_(->}[l]^{ι_4|_{A_4}} \cdots,
}
$$
reproduces all assumptions made in Theorem~\ref{thm:main1}.  By induction
hypothesis, there exists a number $M ∈ \bN^+$ such that the morphisms
$γ_i|_{B_i}$ are étale for all $i ≥ M$.  For simplicity of notation, we may
assume without loss of generality that $M = 1$.  It thus follows immediately
from the characterisation of étale morphisms in terms of fibre size,
\cite[I.~Thm.~10.11]{SGA1}, that for any $i ≥ 1$ the morphism $γ_i$ is étale in
a neighborhood of $B_i$.  Corollary~\ref{cor:WstratBrach} then implies that the
branch locus of any of the composed morphisms $δ_{i,0}$ is contained in the set
$$
T_0' := T_0 \setminus \{\text{union of strata that are dense in some component
  of $T_0$} \},
$$
whose dimension is strictly less than $\dim T_0$.  \qed

\subsection{Proof of Theorem~\ref*{thm:mainSpec}}\label{subsect:proofofCorofMainTheorem}
\approvals{Greb & yes \\ Kebekus & yes \\ Peternell & yes}

Under the assumptions of Theorem~\ref{thm:mainSpec}, set $S := X_{\sing}$, $X_i
:= X$, $ι_i := \Id_X$ and
$$
η_i :=
\begin{cases}
  \Id : Y_0 → X_0 & \text{if $i = 0$}\\
  γ_i ◦ \cdots ◦ γ_1 : Y_i → X_i &\text{if $i > 0$.}
\end{cases}
$$
The spaces and morphisms form a diagram as in~\eqref{eq:fmapseq-1} that
satisfies all assumptions made in Theorem~\ref{thm:main1}.
Theorem~\ref{thm:mainSpec} follows.  \qed

%
%
\svnid{$Id: 07.tex 873 2015-07-28 13:14:15Z kebekus $}

\section{Direct applications}
\label{sec:applPF}
\subversionInfo

We will prove Theorems~\ref{thm:imm1}, \ref{thm:imm1local}, \ref{thm:sidxc} and
\ref{thm:imm4} in this section.

\subsection{Proof of Theorem~\ref*{thm:imm1}}
\label{ssec:pfcim1}

Let $X$ be any variety that satisfies the assumptions of Theorem~\ref{thm:imm1}.
To prove the theorem, we will first show that there exists a quasi-étale Galois
cover $\wtilde X → X$ that satisfies Statement~\ref{il:mc1}.  The equivalence
between \ref{il:mc1} and \ref{il:mc2} is then shown separately.

\subsection*{Step 1: Proof of Statement~\ref*{il:mc1}}
\label{ssec:pfcim11}

Aiming for a contradiction, we assume that given any normal, quasi-projective
variety $\wtilde X$ and any quasi-étale Galois morphism $γ : \wtilde X → X$,
there exists a normal, quasi-projective variety $\what X°$ and an étale cover
$ψ° : \what X° → \wtilde X_{\reg}$ that is not the restriction of any étale
cover of $\wtilde X$.

Given any $\wtilde X$ as above, Theorem~\ref{thm:ZMT-eq} asserts that any étale
cover of $\wtilde X_{\reg}$ extends to a cover of $\wtilde X$.  The assumption
is therefore equivalent to the following.

\begin{assumption}\label{ass:7-1}
  Given any normal, quasi-projective variety $\wtilde X$ and any quasi-étale,
  Galois morphism $γ : \wtilde X → X$, there exists a normal,
  quasi-projective variety $\what X$ and a cover $ψ : \what X → \wtilde X$ that
  is étale over $\wtilde X_{\reg}$, but not étale.
\end{assumption}

Using Assumption~\ref{ass:7-1} repeatedly, and taking Galois closures as in
Theorem~\ref{thm:galoisClosure}, one inductively constructs a sequence of
covers,
$$
\xymatrix{
  X = Y_0 & \ar[l]_(.4){γ_1} Y_1 & \ar[l]_{γ_2} Y_2 & \ar[l]_{γ_3} Y_3 & \ar[l]_{γ_4} \cdots,
}
$$
where all $γ_i$ are quasi-étale, but not étale and where all composed morphisms
$γ_1 ◦ \cdots ◦ γ_i$ are Galois.  We obtain a contradiction to
Theorem~\ref{thm:mainSpec}, showing our initial assumption was absurd.  This
finishes the proof of Statement~\ref{il:mc1}.  \qed

\subsection*{Step 2: Proof of Implication \ref*{il:mc1} $\Rightarrow$ \ref*{il:mc2}}
\label{ssec:pfcim12}

If $\wtilde X → X$ is any cover for which Statement~\ref{il:mc1} holds, we claim
that the push-forward map $\what{ι}_* : \what{π}_1(\wtilde X_{\reg}) →
\what{π}_1(\wtilde X)$ of étale fundamental groups is isomorphic.  For
surjectivity, recall from \Preprint{\cite[0.7.B on p.~33]{FL81} and}
\cite[Prop.~2.10]{Kollar95s} that the push-forward map $ι_*: π_1 \bigl( \wtilde
X^{an}_{\reg}\bigr) → π_1\bigl(\wtilde X^{an} \bigr)$ between topological
fundamental groups is surjective.  Because profinite completion is a right-exact
functor, \cite[Lem.~3.2.3 and Prop.~3.2.5]{RB10}, it follows that $\what{ι}_*$
is likewise surjective.

To show that $\what{ι}_*$ is injective, we use Grothendieck's equivalence
between the category of étale covers and the category of finite sets with
transitive action of the étale fundamental group, \cite[Sect.~5]{Milne80}.
Arguing by contradiction, assume that there exists a non-trivial element $g ∈
\ker \what{ι}_*$.  Since $\what{π}_1(\wtilde X_{\reg})$ is the profinite
completion of $π_1(\wtilde X_{\reg}^{an})$, it is residually finite.  Hence,
there exists a finite group $H$ and a surjective group homomorphism $η :
\what{π}_1(\wtilde X_{\reg}) → H$ such that $η(g) \not = 0$.  By choice of $g$,
the natural action of $\what{π}_1(\wtilde X_{\reg})$ on $H$ is \emph{not}
induced by an action of $\what{π}_1(\wtilde X)$.  Using Grothendieck's
equivalence, we obtain an associated étale cover of $\wtilde X_{\reg}$ that is
not the restriction of any étale cover of $\wtilde X$.  This contradiction to
Statement~\ref{il:mc1} shows that $\what{ι}_*$ is injective and finishes the
proof of Statement~\ref{il:mc2}.  \qed

\subsection*{Step 3: Proof of Implication \ref*{il:mc2} $\Rightarrow$ \ref*{il:mc1}}
\label{ssec:pfcim13}

This is immediate from Grothendieck's equivalence.  The proof of
Theorem~\ref{thm:imm1} is thus finished.  \qed

\subsubsection{Further remarks}
\label{ssec:cormain1gen2}

In the setting of Theorem~\ref{thm:imm1}, the set $\wtilde U :=
γ^{-1}(X_{\reg})$ is a big open subset of the smooth variety
$\wtilde{X}_{\reg}$.  The topological fundamental groups of the complex
manifolds $\wtilde U$ and $\wtilde{X}_{\reg}^{an}$ therefore agree, and the
inclusions $\wtilde U ⊆ \wtilde{X}_{\reg} ⊆ \wtilde{X}$ induce a sequence of
isomorphism between étale fundamental groups
$$
\what{π}_1 \bigl( \wtilde U \bigr) \xrightarrow{\text{isomorphism}} \what{π}_1\bigl( \wtilde{X}_{\reg} \bigr) \xrightarrow{\text{isomorphism}} \what{π}_1\bigl( \wtilde{X} \bigr).
$$

\subsection{Proof of Theorem~\ref*{thm:imm1local}}
\label{ssec2}

The proof of Theorem~\ref{thm:imm1local} follows the argumentation of
Section~\ref{ssec:pfcim11} quite closely, using Theorem~\ref{thm:main1} instead
of the simpler Theorem~\ref{thm:mainSpec}.  \Publication{In order to avoid
  repetition, we omit the proof.  A full, detailed argumentation is found in the
  \href{http://arxiv.org/abs/1307.5718}{arXiv version of this
    paper}.}\Preprint{Maintaining notation and assumptions of
  Theorem~\ref{thm:imm1local}, we argue by contradiction and assume the
  following.
  
  \begin{assumption}\label{ass:x7-0}
    For any open neighbourhood $X°$ of $p$ in $X$ and any quasi-étale, Galois
    morphism $γ : \wtilde X° → X°$, there exists an open neighbourhood $U = U(p)
    ⊆ X°$ with preimage $\wtilde U = γ^{-1}(U)$, and coverings
    \begin{equation}\label{eq:cov}
      \xymatrix{ %
        \what U \ar[rrrr]^{ψ}_{\text{étale over $\wtilde U_{\reg}$, not étale}} &&&& \wtilde U \ar[rrrr]^{γ|_{\wtilde U}}_{\text{\vphantom{$\wtilde U_{\reg}$}quasi-étale, Galois}} &&&& U.  %
      }
    \end{equation}
  \end{assumption}
  
  Using the existence of a Galois closure and the invariance of the branch
  locus, Theorem~\ref{thm:galoisClosure} and \ref{il:GA2}, we are free to assume
  the following in addition.
  
  \begin{assumption}\label{ass:x7-1}
    The composed morphism $γ|_{\wtilde U}◦ψ$ is Galois.
  \end{assumption}
  \CounterStep
  
  Using Assumption~\ref{ass:x7-0}, we will inductively construct an infinite
  diagram of morphisms as in Theorem~\ref{thm:main1},
  \begin{equation}\label{eq:fmapseq-1x}
    \begin{gathered}
      \xymatrix{
	& Y_0 \ar@{->>}[d]_{η_0} & \ar[l]_{γ_1} Y_1 \ar@{->>}[d]_{η_1} & \ar[l]_{γ_2} Y_2 \ar@{->>}[d]_{η_2} & \ar[l]_{γ_3} Y_3 \ar@{->>}[d]_{η_3} & \ar[l]_{γ_4} \cdots \\
        X & \ar@{_(->}[l]^{ι_0} X_0 & \ar@{_(->}[l]^{ι_1} X_1 & \ar@{_(->}[l]^{ι_2} X_2 & \ar@{_(->}[l]^{ι_3} X_3 & \ar@{_(->}[l]^{ι_4} \cdots,
      }
    \end{gathered}
  \end{equation}
  where all $γ_i$ are quasi-étale, but not étale.  This will lead to a
  contradiction when we apply Theorem~\ref{thm:main1} with $S = X_{\sing}$.
  
  \begin{construction}[Construction up to $η_1$]
    Applying Assumption~\ref{ass:x7-0} with $X° := X$ and $γ := \Id_X$, obtain
    an open neighbourhood $U=U(p) ⊆ X°$ and a covering $ψ$ as in \eqref{eq:cov}.
    Set
    \begin{align*}
      X_0 & := U & X_1 & := U & Y_0 & := \wtilde U & Y_1 & := \what U \\
      η_0 & := γ|_{\wtilde U} & η_1 & := γ|_{\wtilde U}◦ψ & γ_1 & := ψ & ι_1 & := \Id_U.
    \end{align*}
    Observe that $η_0$ and $η_1$ are quasi-étale.  The morphism $γ_1$ is
    quasi-étale, but not étale.  Assumption~\ref{ass:x7-1} guarantees that $η_0$
    and $η_1$ are Galois.
  \end{construction}
  
  \begin{construction}[Construction of of $η_{i+1}$]
    Assume that a diagram as in~\eqref{eq:fmapseq-1x} has been constructed, up
    to $η_i$.  Applying Assumptions~\ref{ass:x7-0} and \ref{ass:x7-1} with $X°
    := X_i$ and $γ := η_i$, we obtain an open neighbourhood $U = U(p) ⊆ X_i$ and
    a covering $ψ$ as in \eqref{eq:cov}.  Set
    \begin{align*}
      X_{i+1} & := U & Y_{i+1} & := \what U & γ_{i+1} & := ψ & η_{i+1} := η_i|_{\wtilde U}◦ψ.
    \end{align*}
    Observe that $η_{i+1}$ is quasi-étale.  The morphism $γ_{i+1}$ is
    quasi-étale, but not étale.  Assumption~\ref{ass:x7-1} guarantees that
    $η_{i+1}$ is Galois.
  \end{construction}
  
  In summary, we have obtained a contradiction to Theorem~\ref{thm:main1}
  showing that Assumption~\ref{ass:x7-0} was absurd.  This finishes the proof of
  Theorem~\ref{thm:imm1local}.} \qed

\subsection{Proof of Theorem~\ref*{thm:sidxc}}

Before starting with the proof of Theorem~\ref{thm:sidxc} we note the following
elementary fact which will be used throughout.

\begin{fact}\label{fact:tlbfs}
  Let $X$ be a quasi-projective variety, $S ⊂ X$ any finite set and $\sL$ any
  invertible sheaf on $X$.  Then there exists a Zariski-open set $U ⊆ X$ that
  contains $S$ and trivialises $\sL$, that is, $\sL|_{U} \cong \sO_U$.  \qed
\end{fact}

\subsection*{Step 1: Proof of Statement~\ref*{il:mc1l}}
\label{sssec:mc1l}

Maintaining notation and assumptions of Theorem~\ref{thm:sidxc} and using the
assertions of Theorem~\ref{thm:imm1local}, let $U ⊆ X°$ be any Zariski-open
neighbourhood of $p$ and $\wtilde D$ be any $\bQ$-Cartier Weil divisor on
$\wtilde U$.  Given any point $x ∈ \wtilde U$, we need to show that $\wtilde D$
is Cartier at $x$.

We claim that there exists a Zariski-open subset $V ⊆ U$ that contains $p$ and
$γ(x)$, and a number $n$ such that $\sO_{\wtilde V}(n· D) \cong \sO_{\wtilde
  V}$, where $\wtilde V = γ^{-1}(V)$.  In order to construct $V$, consider the
open, Galois-invariant set $S := γ^{-1}(p) ∪ γ^{-1}\bigl( γ(x) \bigr)$.  By
Fact~\ref{fact:tlbfs}, there exists an open subset $\wtilde W ⊆ \wtilde U$ that
contains $S$ and a number $n$ such that $\sO_{\wtilde W}(n· D) \cong
\sO_{\wtilde W}$.  Set $\wtilde V := \bigcap\nolimits_{g ∈ \Gal(γ)}
g^{-1}(\wtilde W) $.  This is an open subset of $\wtilde W$ that contains $S$,
satisfies $\sO_{\wtilde V}(n· D) \cong \sO_{\wtilde V}$ and is invariant
under the action of the Galois group.  The last point implies that $\wtilde V$
is of the form $γ^{-1}(V)$, for an open subset $V ⊆ U$ that contains $p$ and
$γ(x)$.

Continuing the proof of Statement~\ref{il:mc1l}, choose $n$ minimal, and let $η
: \what V → \wtilde V$ be the associated index-one cover.  The covering map $η$
is quasi-étale and branches exactly over those points of $\wtilde V$ where
$\wtilde D$ fails to be Cartier.  In particular, its restriction
$$
η° := η|_{η^{-1}(\wtilde V_{\reg})} : η^{-1}(\wtilde V_{\reg}) → \wtilde V_{\reg}
$$
is étale.  Recall from Theorem~\ref{thm:imm1local} that $η°$ admits an étale
extension to all of $\wtilde V$.  The uniqueness assertion in Zariski's Main
Theorem, Theorem~\ref{thm:ZMT-eq}, therefore implies that this extension equals
$η$, so that $η$ is itself étale.  As $x ∈ \wtilde V$, this finishes
the proof of Statement~\ref{il:mc1l}.  \qed

\subsection*{Step 2: Proof of Statement~\ref*{il:mc2l}}
\label{sssec:mc2l}

For simplicity of notation, write $G = \Gal(γ)$ and let $m$ denote the size of
this group.  Given any open neighbourhood $U = U(p) ⊆ X°$, any $\bQ$-Cartier
divisor $D$ on $U$ and any point $x ∈ U$, we need to show that $m · D$ is
Cartier at $x$.  Consider the pull-back $\wtilde D := γ^* D$ and recall from
Statement~\ref{il:mc1l} that $\wtilde D$ is Cartier on $\wtilde U$.  Again using
Fact~\ref{fact:tlbfs}, find an open neighbourhood $V$ of $x$ with preimage
$\wtilde V := γ^{-1}(V)$ such that $\wtilde D|_{\wtilde V}$ is linearly
equivalent to zero.  In other words, $\wtilde D|_{\wtilde V} = \divisor (\wtilde
f)$, where $\wtilde f$ is a suitable rational function on $\wtilde V$.  Taking
averages, we obtain a rational function $\wtilde F := \prod\nolimits_{g ∈ G}
\wtilde f ◦ g$, the norm of $\wtilde f$, which is Galois invariant, therefore
descends to a rational function $F$ on $V$, and defines the divisor $\divisor
\wtilde F = m · \wtilde D$.  To finish the proof of Statement~\ref{il:mc2l},
observe that $\divisor F = m · D$.  This also finishes the proof of
Theorem~\ref{thm:sidxc}.  \qed

\subsection{Proof of Theorem~\ref*{thm:imm4}}

We prove Theorem~\ref{thm:imm4} as an application of the following, more general
Proposition~\ref{prop:generalfinitenesscriterion}, which we expect to
have further applications, for example in the classification theory of singular
varieties with trivial canonical class; cf.~\cite[Sect.~8.C]{GKP11}.

\begin{prop}\label{prop:generalfinitenesscriterion}
  Let $\mathcal{R}$ be a set of normal, quasi-projective varieties satisfying
  the following conditions.
  \begin{enumerate}
  \item\label{il:Rx1} If $X ∈ \mathcal{R}$ and if $Y → X$ is any
    quasi-étale Galois cover, then $Y ∈ \mathcal{R}$.
  \item\label{il:Rx2} For each $X ∈ \mathcal{R}$, there exists a $\bQ$-Weil
    divisor $Δ$ such that $(X, Δ)$ is klt.
  \item\label{il:Rx3} Each variety $X ∈ \mathcal{R}$ has finite étale
    fundamental group.
  \end{enumerate}
  Then, for each $X ∈ \mathcal{R}$, the étale fundamental group
  $\what{π}_1(X_{\reg})$ of $X_{\reg}$ is finite.
\end{prop}
\begin{proof}
  Let $X ∈ \mathcal{R}$.  By assumption \ref{il:Rx2}, there exists a $\bQ$-Weil
  divisor $Δ$ such that the pair $(X, Δ)$ is klt.  Let $γ: \wtilde X → X$ be a
  Galois cover with the properties listed in Theorem~\ref{thm:imm1}.  In
  particular, $\what{π}_1(\wtilde X_{\reg}) \simeq \what{π}_1(\wtilde X)$.  We
  obtain an exact sequence as follows,
  $$
  1 → π_1 \Bigl(γ^{-1} \bigl( X_{\reg}^{an} \bigr) \Bigr) → π_1 \bigl( X_{\reg}^{an} \bigr) → \Gal(γ) → 1.
  $$
  Recalling from Section~\ref{ssec:cormain1gen2} that $π_1 \bigl(γ^{-1} (
  X_{\reg}^{an} ) \bigr) = π_1 \bigl(\wtilde X_{\reg}^{an} \bigr)$ and using the
  right-exactness of the profinite completion functor, \cite[Lem.~3.2.3 and
  Prop.~3.2.5]{RB10}, we obtain the following exact sequence of profinite
  completions,
  \begin{equation}\label{eq:coversequence_2}
    \what{π}_1(\wtilde X) → \what{π}_1(X_{\reg}) → \Gal(γ) → 1.
  \end{equation}
  By Assumption~\ref{il:Rx1}, the variety $\wtilde X$ is in $\mathcal{R}$, and
  $\what{π}_1(\wtilde X)$ is hence finite.  Consequently,
  Sequence~\eqref{eq:coversequence_2} presents $\what{π}_1(X_{\reg})$ as an
  extension of two finite groups, which concludes the proof.
\end{proof}

\begin{proof}[Proof of Theorem~\ref*{thm:imm4}]
  For the purpose of this proof, a normal projective variety $X$ is called
  \emph{weakly Fano} if there exists a $\bQ$-Weil divisor $Δ$ such that the pair
  $(X, Δ)$ is klt, and $-(K_X + Δ)$ is nef and big.  We claim that
  $\mathcal{R}=\{\text{weakly Fano varieties} \}$ satisfies the assumptions of
  Proposition~\ref{prop:generalfinitenesscriterion}.  We note that \ref{il:Rx2}
  is trivially satisfied, and that \ref{il:Rx3} follows from a result of
  Takayama \cite[Thm.~1.1]{TakayamaSimpleConnectedness}; see also
  \cite[Cor.~1]{Zh06}, which asserts that any weakly Fano variety is simply
  connected.
  
  In order to show the remaining Property~\ref{il:Rx1}, we let $X$ be weakly
  Fano with boundary divisor $Δ$, and $γ: Y → X$ be a quasi-étale cover.  Use
  finiteness of $γ$ to define a pull-back, $Δ_Y := γ^* (Δ)$, and use finiteness
  of $γ$ as well as \cite[Prop.~5.20]{KM98} to conclude that the pair $(Y, Δ_Y)$
  is klt.  In fact, more is true.  Since $γ$ is quasi-étale, there exists a
  $\bQ$-linear equivalence $K_Y + Δ_Y \sim_{\bQ} γ^* (K_X + Δ)$, which shows
  that $-(K_Y + Δ_Y)$ is nef and big.  In other words, $Y$ is weakly Fano.
\end{proof}

%
%
\svnid{$Id: 08.tex 873 2015-07-28 13:14:15Z kebekus $}

\section{Flat sheaves on klt base spaces}
\subversionInfo
\label{sec:flat}

\subsection{Proof of Theorem~\ref*{thm:flat}}

Recalling the set-up
of Theorem~\ref{thm:flat}, let $X$ be a normal, complex, quasi-projective
variety and assume that there exists a $\bQ$-Weil divisor $Δ$ such that $(X, Δ)$
is klt.  We will prove that there exists a finite, surjective Galois morphism
$γ: \wtilde X → X$, étale in codimension one, such that for any locally free,
flat, analytic sheaf $\sG °$ on $\wtilde{X}_{\reg}^{an}$, there exists a locally
free, flat, analytic sheaf $\sG^{an}$ on $\wtilde X^{an}$ such that
$\sG^{an}|_{X_{\reg}^{an}} \cong \sG°$.  Recalling from \cite[II.5, Cor.~5.8 and
Thm.~5.9]{Deligne70}, that there exists a coherent, reflexive, algebraic sheaf
$\sG$ on $\wtilde X$ whose analytification over $\wtilde{X}_{\reg}$ equals
$\sG°$, the claim will then follow; cf.\ \cite{MR0262386}.

Let $\wtilde X → X$ be any cover for which the assertions of
Theorem~\ref{thm:imm1} hold true.  To shorten notation, we denote the relevant
complex spaces by $Y := \wtilde X^{an}$ and $Y° := \wtilde X^{an}_{\reg}$.  The
inclusion is denoted by $ι: Y° → Y$.  We have seen in Statement~\ref{il:mc2} and
Section~\ref{ssec:cormain1gen2} that the induced morphism of étale fundamental
groups, $\what{ι}_* : \what{π}_1(Y°) → \what{π}_1(Y)$, is isomorphic.

By Definition~\ref{defn:flat}, the sheaf $\sG°$ corresponds to a representation
$ρ°: π_1(Y°) → \GL\bigl(\rank \sF, \bC \bigr)$.  We write $G := \img(ρ°)$.  The
group $G$ is a quotient of the finitely generated group $π_1(Y°)$, hence
finitely generated.  As a subgroup of the general linear group, $G$ is
residually finite by Malcev's theorem, \cite[Thm.~4.2]{MR0335656}.
Consequently, the profinite completion morphism $a : G → \what G$ is injective,
\cite[Sect.~3.2]{RB10}.

To give an extension of $\sG°$ to a flat sheaf on $Y$, we need to show that the
representation $ρ°$ is the restriction of a representation $ρ$ of $π_1(Y)$.
More precisely, we need to find a factorisation
\begin{equation}\label{eq:factx}
  \xymatrix{
    π_1(Y°) \ar[r]_{ι_*} \ar@/^5mm/[rr]^{ρ°} & π_1(Y) \ar[r]_{ρ} & G.
  }
\end{equation}
To this end, recall from \cite[Lem.~3.2.3]{RB10} that taking profinite
completion is functorial.  Hence, we obtain a commutative diagram,
$$
\xymatrix{
  \what G && \ar[ll]_{\what{ρ}°} \what{π}_1(Y°) \ar[rrr]^{\what{ι}_*\text{, isomorphic}} &&& \what{π}_1(Y) \\
  G \ar@{^(->}[u]^{a\text{, inj.}} && \ar[ll]^{ρ°} π_1(Y°) \ar[u]^{b} \ar@{->>}[rrr]_{ι_*\text{, surjective}} &&& π_1(Y), \ar[u]^{c}
}
$$
where all vertical arrows are the natural profinite completion morphisms.  Since
$\what{ι}_*$ is isomorphic by construction, we can set $ρ := \what{ρ}° ◦
(\what{ι}_*)^{-1} ◦ c$.  Recalling from \cite[Prop.~2.10]{Kollar95s} that $ι_*$
is surjective, it follows from commutativity that $\img(ρ) ⊆ \img(a)$.
Identifying $G$ with its image under $a$, we have thus constructed a
factorisation as in \eqref{eq:factx}.  This finishes the proof of
Theorem~\ref{thm:flat}.  \qed

\subsection{Proof of Corollary~\ref{cor:flattangent}}

Let $γ : \wtilde X → X$ be any cover for which the assertion of
Theorem~\ref{thm:flat} holds true.  By assumption, $γ$ is étale in codimension
one, hence étale over $X_{\reg}$.  This has two consequences.  First, consider
the pull-back divisor $Δ_{\wtilde X} := γ^*(Δ)$.  The pair $(\wtilde X,
Δ_{\wtilde X})$ is then klt, \cite[Prop.~5.20]{KM98}.  Second, it follows that
$\sT_{\wtilde X°} \cong γ^* (\sT_{X_{\reg}})$ is a locally free, flat sheaf on
$\wtilde X° := γ^{-1}(X_{\reg})$.  Hence, $\sT_{\wtilde X°}$ admits an extension
 to a locally free, flat sheaf $\wtilde \sF$ on $\wtilde X$.  Since
$\wtilde \sF$ and $\sT_{\wtilde X}$ agree in codimension one and since
$\sT_{\wtilde X}$ is reflexive, both sheaves agree on all of $\wtilde X$.  It
follows that $\sT_{\wtilde X}$ is locally free and flat.  In particular,
$K_{\wtilde X}$ is Cartier, $Δ_{\wtilde X}$ is $\bQ$-Cartier, and the pair
$(\wtilde X, \emptyset)$ is hence klt, \cite[Cor.~2.35]{KM98}.  The solution of
the Lipman-Zariski conjecture for klt spaces,
\cite[Thm.~6.1]{GKKP11}\footnote{See \cite{Druel13a, Gra13} for latest
  results.} thus asserts that $\wtilde X$ is smooth.  Since $γ$ is Galois, hence
a quotient map, $X$ will automatically have quotient singularities.  This proves
the first statement of Corollary~\ref{cor:flattangent}.

Now assume in addition that $X$ is projective.  Then $\wtilde X$ is a smooth,
projective variety with flat tangent bundle and therefore the quotient of an
Abelian variety $A$ by a finite group acting freely, \cite[Chap.~4,
Cor.~4.15]{Kob87}.  Taking the Galois closure of the morphism $A → X$ as in
Theorem~\ref{thm:galoisClosure}, we find a sequence of covers,
$$
\xymatrix{ %
  \wtilde A \ar[rrrr]_{ψ\text{, Galois, étale in codim.~one}} \ar@/^5mm/[rrrrrrrrrr]^{\wtilde{γ}\text{, Galois, étale in codim.~one}} &&&& A \ar[rr]_{\text{étale}} && \wtilde X \ar[rrrr]_{γ\text{, Galois, étale in codim.~one}} &&&& X
}
$$
Since $A$ is smooth, it is clear that the morphism $ψ$, which is \emph{a priori}
étale only in codimension one, is in fact étale.  As an étale cover of an
Abelian variety, $\wtilde A$ is again an Abelian variety.  \qed

%
%
\svnid{$Id: 09.tex 873 2015-07-28 13:14:15Z kebekus $}

\section{Varieties with vanishing Chern classes}
\subversionInfo
\label{sec:c1c2van}

We prove Theorems~\ref{thm:svcc} and \ref{thm:torus} in this section.  The
proofs rely on a boundedness result for families of flat bundles, which we
establish first.

\subsection{Boundedness for families of flat bundles}
\label{ssec:XX}

We will first show that the set of flat bundles forms a bounded family.  Using
the results of Section~\ref{sec:bertini}, this will later allow to recover the
isomorphism type of a bundle from its restriction to a given hyperplane.

\begin{prop}\label{prop:flatbounded}
  Let $X$ be a normal projective variety and $r ∈ \bN^+$ be any integer.  Then,
  the set
  $$
  {\sf B} := \{\sF \mid \sF \text{ a locally free, flat, analytic sheaf on $X$
    with }\rank \sF = r \}
  $$
  is a bounded family.
\end{prop}
\begin{proof}
  If $π: Y → X$ is a resolution of singularities, and $E$ is a locally free,
  flat sheaf on $X$, then $π^*E$ is a locally free, flat, analytic sheaf on $Y$
  such that $π_* π^*E \cong E$.  We may therefore assume without loss of
  generality that $X$ is smooth.  Since any locally free, flat, analytic sheaf
  on $X$ carries an integrable algebraic connection in the sense of
  \cite[p.~24]{MR1320603}, the claim follows from the second part of
  \cite[Thm.~6.13]{MR1320603}.
\end{proof}

\subsection{Proof of Theorem~\ref*{thm:svcc}}
\label{sss:v1}\CounterStep

Let $γ: \wtilde X → X$ be a quasi-étale, Galois cover enjoying the properties
stated in Theorem~\ref{thm:flat}.  Notice first that $\wtilde X$ is still smooth
in codimension two, since $γ$ branches only over the singular set of $X$.
Second, it follows from \cite[Prop.~5.20]{KM98} that $(\wtilde X, \wtilde{Δ} )$
is klt, for $\wtilde{Δ} := γ^* Δ$.  Since $γ$ is finite, the Cartier divisor
$\wtilde H := γ^*(H)$ is ample.  It follows from \cite[Lem.~3.2.2]{MR2665168}
that $\sF := (γ^*\sE)^{**}$ is $\wtilde H$-semistable.
Proposition~\ref{prop:3112} and Assumption~\eqref{eq:svcc} guarantee that
\begin{equation}\label{eq:svcctld}
  c_1(\sF) · \wtilde H^{n-1} = 0, \quad c_1(\sF)² · \wtilde H^{n-2} = 0, \quad\text{and}\quad c_2(\sF) · \wtilde H^{n-2} = 0.
\end{equation}
To prove Theorem~\ref{thm:svcc}, we need to show that $\sF$ is locally free and
flat.

\subsection*{Step 1: Construction of a representation}

It follows from Proposition~\ref{prop:flatbounded} that the family ${\sf B}$ of
locally free, flat, coherent sheaves on $\wtilde X$ whose rank equals $r :=
\rank \sF$ is bounded.  We may thus apply the Mehta-Ramanathan-Theorem for
normal spaces, \cite[Thm.~1.2]{Flenner84}, and Corollary~\ref{cor:bertstabY} to
obtain an increasing sequence of numbers, $0 \ll m_1 \ll m_2 \ll \cdots \ll
m_{n-2}$, as well as general elements $D_i ∈ |\sO_{\wtilde X}(m_i · \wtilde
H)|$ such that the following holds.
\begin{enumerate}
\item The surface $S := D_1 ∩ \cdots ∩ D_{n-2}$ is smooth and contained in
  $\wtilde X_{\reg}$.
\item\label{il:semistable} The restricted sheaf $\sF|_S$ is semistable with
  respect to $\wtilde H|_S$.
\item\label{il:isolift} Let $\sB ∈ {\sf B}$ be any member.  Then, $\sF \cong \sB$
  if and only if $\sF|_S \cong \sB|_S$.
\end{enumerate}
Further, it follows from \eqref{eq:svcctld} and Remark~\ref{rem:altComp} that
\begin{equation}
  c_1(\sF|_S) · (\wtilde H|_S) = 0 \quad \text{and} \quad \ch_2(\sF|_S) = {\textstyle \frac{1}{2}} c_1(\sF|_S)² - c_2(\sF|_S) = 0,
\end{equation}
where $\ch_2$ denotes the second Chern character.  With these equalities, it
follows from Simpson's work, \cite[Cor.~3.10]{MR1179076}, that the semistable
sheaf $\sF|_S$ is flat.  In other words, $\sF|_S$ is given by a representation
$ρ: π_1(S^{an}) → \GL(r, \bC)$.  The Lefschetz Theorem for singular spaces,
\cite[Thm.~in Sect.~II.1.2]{GoreskyMacPherson}, asserts that the natural
homomorphism $ι_*: π_1(S^{an}) → π_1(\wtilde X_{\reg}^{an})$, induced by the
inclusion $ι: S^{an} \into X_{\reg}^{an}$, is isomorphic.  Composing the inverse
$(ι_*)^{-1}$ with $ρ$ we obtain a representation $τ: π_1(\wtilde X_{\reg}^{an})
→ \GL(r, \bC)$.

\subsection*{Step 2: End of proof}

The representation $τ$ defines a flat, locally free, analytic sheaf $\sF_{τ}°$
on $\wtilde X_{\reg}^{an}$, which, by choice of $\wtilde X$, comes from a flat,
locally free, algebraic sheaf $\sF_{τ}$ on $\wtilde X$.  By construction,
$\sF_{τ}|_S$ is isomorphic to $\sF|_S$.  Applying \ref{il:isolift}, we conclude
that $\sF$ is isomorphic to $\sF_{τ}$, and therefore flat.  This finishes the
proof of Theorem~\ref{thm:svcc}.  \qed

\subsubsection{Concluding remarks}

The results of this section clearly remain true if we substitute the
polarisation $(H, …, H)$ by $H_1, …, H_{n-1}$, where the $H_j$ are not
necessarily identical ample divisors.

If $\sE$ is polystable, it is possible to avoid the use of
\cite[Cor.~3.10]{MR1179076} in the proof of Theorem~\ref{thm:svcc}, by using
\cite[Thm.~1]{Donaldson85} to conclude that $\sF|_S$ is Hermite-Einstein.  This
suffices for the proof of Theorems~\ref{chern} and \ref{thm:torus} below, since
---in the setting of Theorem~\ref{thm:torus}--- the tangent sheaf is polystable
for any polarisation by \cite[Cor.~7.3]{GKP11}, possibly after a quasi-étale
cover of $X$.

\subsection{Proof of Theorem~\ref*{thm:torus} }
\label{sect:second}

Using the terminology introduced in Section~\ref{sec:ChernSing} we state the
main result of this section.  Theorem~\ref{thm:torus} follows directly from
this.

\begin{thm}[Characterisation of quotients of Abelian varieties]\label{chern}
  Let $X$ be a normal $n$-dimensional projective variety which is klt and smooth
  in codimension two.  Then the following conditions are equivalent:
  \begin{enumerate}
  \item\label{torus1} $K_X \equiv 0$, and $c_2(\sT_X)$ is numerically trivial in
    the sense of Definition~\ref{def:B}.
  \item\label{torus2} There exists an Abelian variety $A$ and a finite,
    surjective, Galois morphism $A → X$ that is étale in codimension two.
  \end{enumerate}
\end{thm}
\begin{rem}[Comparing the assumptions of Theorem~\ref{chern} and Theorem~\ref{thm:torus}]
  By Proposition~\ref{prop:trivialityCriterion}, the conditions in \ref{torus1}
  are equivalent the following conditions: \emph{$K_X \equiv 0$, and there exist
    ample divisors $H_1, …, H_{n-2}$ on $X$ such that $c_2(\sT_X) · H_1
    \cdots H_{n-2} = 0$}.  This establishes the link between Theorem~\ref{chern}
  and Theorem~\ref{thm:torus}, as stated in the introduction.
\end{rem}
\begin{rem}
  In dimension three, Theorem~\ref{chern} has been shown without the assumption
  on the codimension of the singular set by Shepherd-Barron and Wilson,
  \cite{SBW94}, using orbifold techniques.  It seems feasible to obtain an
  analogous result also in higher dimensions with the methods presented here.
\end{rem}

Theorem~\ref{chern} is shown in the remainder of this section.  The two
implications will be shown separately.

\subsection*{Proof of Theorem~\ref*{chern}, \ref*{torus2} $\Rightarrow$ \ref*{torus1}}
\label{ssec:hin}

If $η: A → X$ is any finite map from an Abelian variety, étale in codimension
two, then there exists a linear equivalence $η^*(K_X) \sim K_A = 0$.  In
particular, it follows from the projection formula that $K_X$ is numerically
trivial.  Corollary~\ref{cor:311} implies that $c_2(\sT_X)$ is numerically
trivial as well.  \qed

\subsection*{Proof of Theorem~\ref*{chern}, \ref*{torus1} $\Rightarrow$ \ref*{torus2}}
\label{ssec:7BN}

Assume that \ref{torus1} holds.  We first reduce to the case of canonical
singularities.  By the abundance theorem for klt varieties with numerically
trivial canonical divisor class, \cite[Thm.~4.2]{AmbroLCtrivial}, there exists
an $m ∈ \bN^+$ such that $\sO_X(m K_X) \cong \sO_X$.  Let $m$ be minimal with
this property, and let $\nu: \what X → X$ be the associated global index-one
cover, which is quasi-étale over $X$.  Then, $\sO_{\wtilde X}(K_{\wtilde X})
\cong \sO_{\wtilde X}$, which together with the fact that $\wtilde X$ has klt
singularities, \cite[Prop.~5.20]{KM98}, implies that $\wtilde X$ has canonical
singularities.  Moreover, applying Corollary~\ref{cor:311} we see that $\wtilde
X$ is smooth in codimension two, and that $c_2(\sT_X)$ is numerically trivial.
If $η: \wtilde A → \wtilde X$ is a finite, surjective, Galois morphism that is
étale in codimension two from an abelian variety $\wtilde A$ to $\wtilde X$,
then taking the Galois closure of $\nu ◦ η$ yields the desired map $A → X$.  We
may therefore make the following simplifying assumption.

\begin{asswlog}\label{ref:cansings}
 The variety $X$ has canonical singularities.
\end{asswlog}

Since $c_2(X) $ is numerically trivial, we have $c_2(X) · H^{n-2} = 0$ for
any ample divisor $H$ on $X$.  Since moreover $K_X$ is numerically trivial, and
since $X$ has at worst canonical singularities by Assumption~\ref{ref:cansings},
the tangent sheaf $\sT_X$ is $H$-semistable by \cite[Prop.~5.4]{GKP11}, and
$$
c_1(\sT_X) · H^{n-1} = c_1(\sT_X)² · H^{n-2}=0 \quad \text{for any ample divisor $H$ on $X$.}
$$
By Theorem~\ref{thm:svcc}, there hence exists a quasi-étale cover $γ: \wtilde X
→ X$ such that $\sT_{\wtilde X} \cong \bigl( γ^* \sT_X \bigr)^{**}$ is a flat,
locally free sheaf.  Applying Corollary~\ref{cor:flattangent} to $\wtilde X$ and
possibly taking Galois closure finishes the proof of Theorem~\ref{chern}.  \qed

%
%
\svnid{$Id: 10.tex 873 2015-07-28 13:14:15Z kebekus $}

\section{Varieties admitting polarised endomorphisms}
\subversionInfo
\label{sec:NZ}

We will prove Theorem~\ref{thm:NZ} in Section~\ref{ssec:pftnz}.  As noted in the
introduction, Theorem~\ref{thm:NZ} has consequences for the structure theory of
varieties with endomorphisms.  These are discussed in
Section~\ref{ssec:structure} below.

\subsection{Proof of Theorem~\ref*{thm:NZ}}
\label{ssec:pftnz}

Theorem~\ref{thm:NZ} has been shown in \cite[Thm.~3.3]{MR2587100} under an
additional assumption concerning fundamental groups of smooth loci of
Euclidean-open subsets of $X$.  Nakayama and Zhang use this assumption only
once, to prove a claim which appears on Page~1004 of their paper.  After briefly
recalling their setup, we will show that the claim follows directly from our
Theorem~\ref{thm:mainSpec}, without any additional assumption.  Once this is
done, the original proof of Nakayama--Zhang applies verbatim.

Under the assumptions of Theorem~\ref{thm:NZ}, let $f : X → X$ be a polarised
endomorphism.  Nakayama--Zhang start the proof of \cite[Thm.~3.3]{MR2587100} on
page 1004 of their paper by recalling from \cite[Thm.~3.2]{MR2587100} that $X$
has at worst canonical singularities, that $K_X$ is $\bQ$-linearly trivial, and
that $f$ is étale in codimension one.  Given any index $k ∈ \bN^+$, they
consider the iterated endomorphism $f^k$ and its Galois closure,
$$
\xymatrix{
  V^k \ar[r]_{τ_k} \ar@/^4mm/[rr]^{θ^k} & X \ar[r]_{f^k} & X,
}
$$
where $θ_k$ and $τ_k$ are Galois and again étale in codimension one;
cf.~Theorem~\ref{thm:galoisClosure}.  By \cite[Lem.~2.5]{MR2587100}, there exist
finite morphisms $g_k$, $h_k$, again étale in codimension one, forming
commutative diagrams as follows,
$$
\begin{gathered}
  \xymatrix{ %
    & V_1 \ar[d]_{τ_1} \ar@/_3mm/[dl]_{θ_1} & \ar[l]_{h_1} V_2 \ar[d]_{τ_2} & \ar[l]_{h_2} \cdots\\
    X & \ar[l]^{f} X & \ar[l]^{f} X & \ar[l]^{f} \cdots
  }
\end{gathered}
\quad\text{and}\quad
\begin{gathered}
  \xymatrix{ %
    V_1 \ar[d]_{τ_1} & \ar[l]_{g_1} V_2 \ar[d]_{τ_2} & \ar[l]_{g_2} V_3 \ar[d]_{τ_3} & \ar[l]_{g_3} \cdots\\
    X \ar@{=}[r]_{\phantom{f}} & X \ar@{=}[r] & X \ar@{=}[r] & \cdots.
  }
\end{gathered}
$$
Nakayama--Zhang claim in \cite[claim on p.~1004]{MR2587100} that the morphisms
$h_k$ and $g_k$ are étale for all sufficiently large $k$.  They prove this claim
using their additional assumption on the fundamental groups.
Theorem~\ref{thm:mainSpec}, however, applies to the associated sequences,
$$
\xymatrix{ %
  X & \ar[l]^{θ_1} V_1 & \ar[l]^{h_1} \ar@/_4mm/[ll]_{θ_2\text{, Galois}} V_2 & \ar[l]^{h_2} \cdots \\
}
\quad\text{and}\quad
\xymatrix{ %
  X & \ar[l]^{τ_1} V_1 & \ar[l]^{g_1} \ar@/_4mm/[ll]_{τ_2\text{, Galois}} V_2 & \ar[l]^{g_2} \cdots,\\
}
$$
and yields this result without any extra assumption.  As pointed out above, the
rest of Nakayama--Zhang's proof applies verbatim.  \qed

\subsection{The structure of varieties admitting endomorphisms}
\label{ssec:structure}

Theorem~\ref{thm:NZ} has consequences for the structure theory of varieties with
endomorphisms.  The following results have been shown in
\cite[Thm.~1.3]{MR2587100}, conditional to the assumption that
\cite[Conj.~1.2]{MR2587100} = Theorem~\ref{thm:NZ} holds true.  The definition
$q^{\sharp}$ is recalled in Remark~\ref{rem:qsharp} below.

\begin{thm}[Structure of varieties admitting polarized endomorphisms]\label{thm:NZmain}
  Let $f: X → X$ be a non-isomorphic, polarised endomorphism of a normal,
  complex, projective variety $X$ of dimension $n$.  Then $κ(X) ≤ 0$ and
  $q^\sharp(X,f) ≤ n$.  Furthermore, there exists an Abelian variety $A$ of
  dimension $\dim A = q^\sharp(X,f)$ and a commutative diagram of normal,
  projective varieties,
  $$
  \xymatrix{ %
    A \ar[d]_{f_A} && Z \ar[d]_{f_Z} \ar[ll]_{ω} \ar[rr]^{ρ} && V \ar[d]_{f_V} \ar[rrrr]^{τ} &&&& X \ar[d]^{f} \\
    A && Z \ar[ll]_{ω}^{\text{flat, surjective}} \ar[rr]^{ρ}_{\text{biratl.}} && V \ar[rrrr]^{τ}_{\text{finite, surjective, étale in codim.~one}} &&&& X,
  }
  $$
  where all vertical arrows are polarised endomorphism, and every fibre of $ω$
  is irreducible, normal and rationally connected.  In particular, $X$ is
  rationally connected if $q^\sharp(X,f) = 0$.

  Moreover, the fundamental group $π_1(X)$ contains a finitely generated,
  Abelian subgroup of finite index whose rank is at most $2 ·
  q^\sharp(X,f)$.  \qed
\end{thm}

\begin{rem}[\protect{Definition of $q^\sharp$, \cite[p.~992f]{MR2587100}}]\label{rem:qsharp}
  In the setting of Theorem~\ref{thm:NZmain}, the number $q^\sharp(X,f)$ is
  defined as the supremum of irregularities $q(\wtilde X') = h¹ \bigl(\wtilde
  X',\, \sO_{\wtilde X'} \bigr)$ of a smooth model $\wtilde X'$ of $X'$ for all
  the finite coverings $τ : X' → X$ étale in codimension one and admitting an
  endomorphism $f' : X' → X'$ with $τ◦f'=f◦τ$.
\end{rem}

%
%
\svnid{$Id: 11.tex 873 2015-07-28 13:14:15Z kebekus $}

\section{Examples, counterexamples, and sharpness of results}
\subversionInfo
\label{sec:ex}

In this section, we have collected several examples which illustrate to what
extend our main results are sharp.  In Section~\ref{subsect:Kummer} we construct
an infinite sequence of branched non-Galois coverings of the singular Kummer
surface, showing that Theorem~\ref{thm:mainSpec} does not to hold without the
Galois assumption.  Section~\ref{ssec:GZ} discusses an example of Gurjar--Zhang,
showing that Theorem~\ref{thm:imm4} is sharp, and that
Theorem~\ref{thm:mainSpec} has no simple reformulation in terms of the
push-forward between fundamental groups.  Finally, Section~\ref{sec:noMinTildeX}
shows by way of an example that there is generally no canonical, minimal choice
for the coverings constructed in Theorem~\ref{thm:imm1}.

\subsection{Isogenies of Abelian surfaces and the Kummer construction}
\label{subsect:Kummer}

We first construct a number of special endomorphisms of Kummer surfaces.  For
this, fix one Abelian surface $A$ throughout this section.

\begin{construction}[Endomorphisms of singular Kummer surfaces]\label{const:sKs}
  Consider the involutive
  automorphism $σ: A → A$, $a \mapsto -a$.  Set $X := A/σ$ and denote the
  quotient morphism by $f: A → X$, $a \mapsto [a]$.  The automorphism $σ$ has
  $16$ fixed points.  The quotient surface $X$ has $16$ rational double points,
  which are canonical and therefore klt.  We call $X$ a ``singular Kummer
  surface''.

  Given any number $n ∈ \bN^+$, consider the endomorphism of $A$ obtained by
  multiplication with $n$, that is, $\mathbf{n}_A: A → A$, $a \mapsto n· a$.
  Note that $\mathbf{n}_A$ is a finite morphism that commutes with $σ$ and
  therefore induces a finite endomorphism of $X$, which we denote as
  $\mathbf{n}_X: X → X$, $[a] \mapsto [n · a]$.  By construction, the
  morphism $\mathbf{n}_X$ commutes with $f$.
\end{construction}

The main properties of the endomorphisms $\mathbf{n}_A$ and $\mathbf{n}_X$ are
summarised in the following observations.

\begin{obs}[Degree and fibres of $\mathbf{n}_A$ and $\mathbf{n}_X$]\label{obs:nx2}
  The morphism $\mathbf{n}_A$ is the quotient for the natural action of the
  $n$-torsion group $A_n < A$ on $A$.  In particular, $\mathbf{n}_A$ is finite
  of degree $4^n$, étale and Galois.  Given any point $x ∈ X$, it follows from
  commutativity, $\mathbf{n}_X ◦ f = f ◦ \mathbf{n}_A$, that the number of
  points in the fibre is at least $\# \: \mathbf{n}_X^{-1} \bigl( x \bigr) ≥
  \frac{1}{2} · 4^n$.
\end{obs}

\begin{obs}[Endomorphism $\mathbf{n}_X$ is quasi-étale, not étale, and not Galois]\label{obs:nx3}
  The morphism $\mathbf{n}_X$ is quasi-étale by construction.  Let $n > 2$ be
  any number and $x ∈ X_{\sing}$ any singular point of $X$.  Since $X$ has only
  16 singularities, it follows from Observation~\ref{obs:nx2} that the fibre
  $\mathbf{n}_X^{-1}\mathbf{n}_X(x)$ contains both singular and non-singular
  points of $X$.  This shows that $\mathbf{n}_X$ is neither étale nor Galois if
  $n > 2$.
\end{obs}

Construction~\ref{const:sKs} shows that Theorem~\ref{thm:mainSpec} does not hold
without the Galois assumption.  This is the content of the following
proposition.

\begin{prop}[Sharpness of Theorem~\ref{thm:mainSpec}]\label{prop:msharp}
  Using the notation of Construction~\ref{const:sKs}, choose a number $n > 2$
  and consider the tower of finite morphisms
  $$
  \xymatrix{ %
    X & \ar[l]_{\mathbf{n}_X} X & \ar[l]_{\mathbf{n}_X} X &
    \ar[l]_{\mathbf{n}_X} X & \ar[l]_{\mathbf{n}_X} \cdots.  }
  $$
  Then, all assumptions of Theorem~\ref{thm:mainSpec} are satisfied except that
  the composed morphisms $(\mathbf{n^k})_X = \mathbf{n}_X ◦ \cdots ◦
  \mathbf{n}_X$ be Galois, for any $k > 1$.  However, none of the finite
  morphisms $\mathbf{n}_X$ is étale.  \qed
\end{prop}

\subsection{Comparing the étale fundamental groups of a klt variety and its smooth locus}
\label{ssec:GZ}

In \cite[Sect.~1.15]{MR1298542}, Gurjar and De-Qi Zhang construct a rationally
connected, simply-connected, projective surface $X$ with rational double point
singularities, admitting a quasi-étale, two-to-one cover $γ: \wtilde X → X$
where $\wtilde X$ is smooth and $π_1(\wtilde{X}^{an}) \cong \bZ²$.

We claim that $\what{π}_1(X_{\reg})$ is infinite.  In fact, the standard
sequence
$$
0 → \underbrace{π_1\bigl( γ^{-1}(X_{\reg}^{an}) \bigr)}_{\cong
  π_1(\wtilde{X}^{an})} \xrightarrow{γ_*} π_1(X_{\reg}^{an}) → \bZ/2\bZ → 0
$$
shows that the groups $n \bZ²$ are finite-index subgroups of
$π_1(X_{\reg}^{an})$, for all $n ∈ \bN^+$.  Recalling from the definition that
the kernel of the profinite completion map $π_1(X_{\reg}^{an}) →
\what{π}_1(X_{\reg})$ equals the intersection of all finite-index subgroups, the
claim follows.

The example of Gurjar--Zhang relates to our results in at least two ways.

\subsubsection*{Morphism of fundamental groups in Theorem~\ref*{thm:mainSpec}}

Given a quasi-projective, klt variety $X$, we have pointed out in
Remark~\ref{rem:p1hat} that Theorem~\ref{thm:mainSpec} does \emph{not} assert
that the kernel of the natural surjection of étale fundamental groups,
$\what{ι}_*: \what{π}_1(X_{\reg}) \rightarrow \what{π}_1(X)$, is finite.  The
example of Gurjar--Zhang proves this point.  Note that the singular Kummer
surface constructed in Section~\ref{subsect:Kummer} above also exemplifies the
same point.

\subsubsection*{Sharpness of Theorem~\ref*{thm:imm4}}

It is well-known that rationally chain connected varieties have finite
fundamental group, \cite[Thm.~4.13]{Kollar95s}.  Hence, it is a natural question
whether the condition ``$-(K_X + Δ)$ is nef and big'' of Theorem~\ref{thm:imm4}
can be weakened to ``$X$ rationally chain connected''.  The example of
Gurjar--Zhang shows that this is not the case.

\subsection{Non-existence of a minimal choice in Theorem~\ref*{thm:imm1}}
\label{sec:noMinTildeX}

We will show in this section that, given a klt variety $X$, there is in general
no canonical choice for the cover $\wtilde X$ constructed in
Theorem~\ref{thm:imm1}, that is minimal in the sense that any other such cover
would factorise over $\wtilde X$.  For this, we will construct a
quasi-projective klt surface $X$ and two covers $γ_1 : \wtilde X_1 → X$, $γ_2 :
\wtilde X_2 → X$ that satisfy the conditions of Theorem~\ref{thm:imm1}.  The
geometry of these covers will differ substantially.  It will be clear from the
construction that the two covers do not dominate a common third.  In other
words, there is no cover $ψ : Y → X$ that satisfies the conclusions of
Theorem~\ref{thm:imm1} and fits into a diagram as follows,
\begin{equation}\label{eq:nexdiag}
  \begin{gathered}
    \xymatrix{
      \wtilde X_1 \ar[r]^{α_1} \ar@/_2mm/[dr]_{γ_1} & Y \ar[d]_{ψ} & \wtilde X_2 \ar[l]_{α_2} \ar@/^2mm/[dl]^{γ_2} \\
	& X.
	}
  \end{gathered}
\end{equation}
This shows that a canonical, minimal covering cannot exist.

\subsection*{Construction of the surface $X$}
\label{ssec:conX}

Consider the spaces $A = \bP¹ ⨯ \bP¹$ and $B := \bP¹$.  Identifying the group
$\bZ_2$ with the multiplicative group $\{-1,1\}$, consider the actions of the
Klein four-group $G := \bZ_2 ⨯ \bZ_2$ on $A$ and $B$, written in inhomogeneous
coordinates as follows
\begin{align*}
  \bigl( \bZ_2 ⨯ \bZ_2 \bigr) ⨯ A & → A, & \bigl( (z_1,z_2),\,(a_1, a_2)\bigr) & \mapsto \bigl(z_1 · a_1^{z_2},\, z_1z_2 · a_2 \bigr) \\
  \bigl( \bZ_2 ⨯ \bZ_2 \bigr) ⨯ B & → B, & \bigl( (z_1,z_2),\, b \bigr) & \mapsto (z_1 · b^{z_2}).
\end{align*}
Let $π : A → B$ be the projection to the first factor.  This map is clearly
equivariant and therefore induces a morphism between quotients,
$$
\xymatrix{ %
  A \ar[d]_{π} \ar[rrrr]^(.4){q_A\text{, quotient, four-to-one}} &&&& A/G \hphantom{\cong \bP¹} \ar@<-4mm>[d]^{ρ}\\
  B \ar[rrrr]_(.4){q_B\text{, quotient, four-to-one}} &&&& B/G \cong \bP¹
  }
$$
The quotient curve $B/G$ is isomorphic to $\bP¹$.  The quotient map $q_B$ has
three branch points, $\{q_1, q_2, q_3\} ⊂ B/G$, and six ramification points $\{
0, \infty, \pm 1, \pm i \} ⊂ B$, two over each of the branch points, and each
with ramification index two.

The quotient surface $A/G$ is singular, with six quotient singularities of type
$A_1$, two over each of the branch points $q_i ∈ B/G$.  The quotient $A/G$ is
therefore klt.  Away from the branch points $q_i$, the morphism $ρ$ has the
structure of a $\bP¹$-bundle.  The fibres $ρ^{-1}(q_i)$ are set-theoretically
isomorphic to $\bP¹$, but carry a non-reduced, double structure.  Choose a
fourth point $q_4 ∈ B/G$ and consider the quasi-projective variety $X := (A/G)
\setminus ρ^{-1}(q_4)$.  \Publication{The preprint version of this paper
  contains a figure which depicts the situation.}\Preprint{The setup is depicted
  in Figure~\ref{fig:gkl}.
  
  \begin{figure}[t]
    \centering
    \begin{tikzpicture}[scale=.9]
      \fill[fill=gray!20!white] (-2.5, 1.5) node[above] {$X$} rectangle (2.5, -1.5);
      
      \draw (-1.5-0.02, 2) to (-1.5-0.02,-2);
      \draw (-1.5+0.02, 2) to (-1.5+0.02, -2);
      \fill (-1.5, 1) node{$\ast$};
      \fill (-1.5, -1) node{$\ast$};
      
      \draw (-0.5-0.02, 2) to (-0.5-0.02,-2);
      \draw (-0.5+0.02, 2) to (-0.5+0.02, -2);
      \fill (-0.5+0, 1) node{$\ast$};
      \fill (-0.5+0, -1) node{$\ast$};
      
      \draw (0.5-0.02, 2) to (0.5-0.02,-2);
      \draw (0.5+0.02, 2) to (0.5+0.02, -2);
      \fill (0.5, 1) node{$\ast$};
      \fill (0.5, -1) node{$\ast$};
      
      \draw [->] (1.7, 2) node[above]{\scriptsize fibre removed} -- (1.5,1.6);
      \draw[fill=white, draw=white] (1.5-.03, -1.5) rectangle (1.5+.03, 1.5);
      \draw[dashed] (1.5-.03, -1.5) to (1.5-.03, 1.5);
      \draw[dashed] (1.5+.03, -1.5) to (1.5+.03, 1.5);
      
      \draw[dashed] (-2.5, -1.5) rectangle (2.5, 1.5);
      
      \draw [->] (0, -2) -- node[right]{$ρ$} (0, -3);
      
      \draw (-2.5, -3.5) node[above]{$B/G \cong \bP¹$} -- (2.5, -3.5);
      \filldraw (-1.5,-3.5) circle(0.05) node[below]{\scriptsize $q_1$};
      \filldraw (-0.5,-3.5) circle(0.05) node[below]{\scriptsize $q_2$};
      \filldraw (0.5,-3.5) circle(0.05) node[below]{\scriptsize $q_3$};
      \filldraw (1.5,-3.5) circle(0.05) node[below]{\scriptsize $q_4$};
    \end{tikzpicture}
    
    {\small The figure shows the singular quotient surface $X$ that is constructed
      in Section~\ref{ssec:conX}.  The surface contains six $A_1$ quotient
      singularities which are indicated with the symbol ``$\ast$''.}
    
    \caption{Singular surface with $A_1$ singularities}
    \label{fig:gkl}
  \end{figure}
}

\subsection*{Construction of the covering $\wtilde X_1$}
\label{ssec:conX1}

Set $\wtilde X_1 := q_A^{-1}(X)$ and $γ_1 := q_A|_{\wtilde X_1}$.  Since
$\wtilde X_1$ is smooth, the condition that $\what{π}_1 \bigl( \wtilde X_1
\bigr) = \what{π}_1 \bigl( (\wtilde X_1)_{\reg} \bigr)$ is satisfied.  The cover
$γ_1$ is Galois, four-to-one, and branches only over the singularities of $X$.

\subsection*{Construction of the covering $\wtilde X_2$}
\label{ssec:conX2}

Consider a double covering of $B/G \cong \bP¹$ by an elliptic curve $E$,
branched exactly over the points $q_1, …, q_4$.  Let $\wtilde X_2$ be the
normalisation of the fibred product $X ⨯_{B/G} E$ and let $γ_2: \wtilde X_2 → X$
be the obvious morphism, which is a Galois, two-to-one cover of $X$, branched
exactly over the singularities of $X$.  A standard computation shows that
$\wtilde X_2$ is smooth.  The condition that $\what{π}_1 \bigl( \wtilde X_2
\bigr) = \what{π}_1 \bigl( (\wtilde X_2)_{\reg} \bigr)$ is thus again trivially
satisfied.

\subsection*{Non-existence of a minimal cover}

We will now show that a covering $ψ : Y → X$ forming a diagram as in
\eqref{eq:nexdiag} cannot exist.  Assuming for a moment that $Y$ \emph{does}
exist, recall that $γ_2$ is two-to-one.  It follows that either $α_2$ or $ψ$ is
isomorphic.  Neither is possible: If $α_2$ was isomorphic, we would obtain a
two-to-one covering map $α_2^{-1}◦α_1 : \wtilde X_1 → \wtilde X_2$.  This is
impossible, because $\wtilde X_1$ is rational, while $\wtilde X_2$ is not.  If
$ψ$ was isomorphic, then $X$ would have to satisfy the condition that
$\what{π}_1 \bigl( X \bigr) = \what{π}_1 \bigl( X_{\reg} \bigr)$.  The existence
of covering maps that branch only over the singularities shows that this is not
the case.

\Preprint{
\part*{Appendices}
\appendix

%
%
\svnid{$Id: appendix.tex 873 2015-07-28 13:14:15Z kebekus $}

\section{Zariski's Main Theorem in the equivariant setting}
\label{appendix:A}
\subversionInfo

While certainly known to experts, we were not able to find a full reference for
Zariski's Main Theorem in the equivariant setting, Theorem~\ref{thm:ZMT-eq}, in
the literature.  The following lemma will be used.

\begin{lem}\label{lem:A2}
  Consider the composition of a rational map $a : V \dasharrow W_1$ and a finite
  morphism $b: W_1 → W_2$, where $V$, $W_1$ and $W_2$ are quasi-projective
  varieties.  If $V$ is normal and if $b◦a$ is a morphism, then $a$ is a
  morphism.
\end{lem}
\begin{proof}
  Blowing up $V$ suitably, we obtain a diagram as follows
  $$
  \xymatrix{ %
    \wtilde V \ar@{->>}[rr]_{β_V\text{, blow-up}} \ar@/^4mm/[rrr]^{β_W\text{, morphism}} && V \ar@{-->}[r]_(.4){\vphantom{b}a} & W_1 \ar[r]_(.45)b & W_2.
  }
  $$
  Since $V$ is normal, it suffices to show that given any closed point $v ∈ V$
  with fibre $F := β_V^{-1}(v)$, the image set $β_W(F)$ is a point.  To this
  end, recall Zariski's main theorem, \cite[III~Cor.~11.4]{Ha77}, which asserts
  that the fibre $F$ is connected.  Consequently, so is $β_W(F)$.  Using the
  assumption that $b◦a$ is a morphism, observe that
  $$
  b \bigl( β_W(F) \bigr) = (b◦a) \bigl( β_V(F) \bigr) = (b◦a)(v) = \text{ a point.}
  $$
  Since $b$ is finite, it follows that the connected set $β_W(F)$ is a point, as
  claimed.
\end{proof}

\subsection{Proof of Theorem~\ref*{thm:ZMT-eq}}
\label{sec:pfzmt}

Assume that we are given a morphism $a : V° → W$ as in Theorem~\ref{thm:ZMT-eq}.
The claims made in the theorem will be shown separately.

\subsection*{Step 1: Existence of a factorisation}

Zariski's Main Theorem in the form of Grothendieck, \cite[Thm.~8.12.6]{EGA4-3},
asserts the existence of a quasi-projective variety $V'$ and a factorisation of
$a$ into an open immersion $α' : V° → V'$ and a finite morphism $β' : V' → W$.
Let $η : V → V'$ be the normalisation.  Recalling the universal property of the
normalisation map, \cite[II Ex.~3.8]{Ha77}, observe that there exists a morphism
$α : V° → V$ such that $α' = α ◦ η$.  Finally, set $β := β' ◦ η$.

\subsection*{Step 2: Uniqueness of the factorisation}

Assume we are given two factorisations of $a$ as in Theorem~\ref{thm:ZMT-eq},
say $a = β_1 ◦ α_1 = β_2 ◦ α_2$.  We obtain a commutative diagram of morphisms
and rational maps,
$$
\xymatrix{ %
  V° \ar[rrr]^{\phantom{β}α_1\text{, open immersion}\phantom{β}} \ar@{=}[d] &&& V_1 \ar[rr]^{β_1\text{, finite}} \ar@{-->}[d]_{\txt{\scriptsize $r := α_2 ◦ α_1^{-1}$\\\scriptsize birational}} && W \ar@{=}[d] \\
  V° \ar[rrr]_{\phantom{β}α_2\text{, open immersion}\phantom{β}} &&& V_2 \ar[rr]_{β_2\text{, finite}} && W.
}
$$
Since the composed morphism $β_2 ◦ r$ equals $β_1$, Lemma~\ref{lem:A2} asserts
that $r$ is a morphism.  The same holds for $r^{-1} = α_1 ◦ α_2^{-1}$, showing
that $r$ is isomorphic.

\subsection*{Step 3: The \boldmath$G$-action on \boldmath$V$}

To show that $G$ acts on $V$, let $g ∈ G$ be any element.  The action of $g$
gives rise to a diagram
$$
\xymatrix{ %
  V \ar@{-->}[r]^{\overline{g}} & V \ar[r]^{β} & W \\
  V° \ar[r]_g \ar[u] & V° \ar[u] \ar[r]_{a} & W°, \ar[u] }
$$
where all vertical arrows are the natural inclusion maps and where
$\overline{g}$ is the natural extension of $g$ to a birational endomorphism of
$V$.  The morphism $a$ is $G$-invariant, which means that $a◦g = g◦a$.  As a
consequence, we see that $β◦\overline{g} = g◦β$.  In particular, the rational
map $β◦\overline{g}$ is a morphism.  Lemma~\ref{lem:A2} thus implies that
$\overline{g}$ is a morphism, showing that $G$ acts on $V$, as claimed.
Equivariance of $α$ and $β$ is now automatic.

\subsection*{Step 4: Quotients}

Under the assumptions of \ref{il:zmt3}, we have seen that $G$ acts on $V$.  We
need to show that $W$ is the quotient of the $G$-action on $V$ and that $β$ is
the quotient map.  To this end, consider the following natural diagram of
(normal) quotient spaces and quotient maps,
$$
\xymatrix{ %
  V \ar[rr]_{\text{quotient}} \ar@/^5mm/[rrrr]^{β\text{, finite and $G$-invariant}} && \ar[rr]_{q} V/G && W \\
  V° \ar[u] \ar[rr]_{\text{quotient}} && V°/G \ar[u] \ar@{<->}[rr]_(.55){\text{isomorphism}} && W°, \ar@{^(->}[u]_{\text{inclusion}}
}
$$
where all vertical arrows are the natural inclusion maps.  The map $q$ is
induced by the universal property of the quotient, \cite[Thm.~1 on
p.~111]{MR2514037}, using that the map $β$ is $G$-invariant.  The isomorphism
between $V°/G$ and $W°$ shows that $q$ is birational.

Since $β$ is finite, so is $q$.  The standard version of Zariski's Main Theorem,
\cite[V~Thm.~5.2]{Ha77}, therefore applies.  It asserts that $V/G$ is isomorphic
to $W$ and that $β$ is the quotient map.  This finishes the proof of
Theorem~\ref{thm:ZMT-eq}.  \qed

\section{Galois closure}
\label{appendix:B}

\subsection{Proof of Theorem~\ref*{thm:galoisClosure}}

We will use the equivalence between the category of connected, étale covers of a
variety $V$ and the category of finite sets $S$ with transitive action of the
étale fundamental group $\what{π}_1(V)$, \cite[Sect.~5 and Thm.~5.3]{Milne80}.

\subsection*{Step 1: Construction of $\wtilde X$}

Consider the maximal open set $Y° := Y \setminus \Br(γ)$ where $γ$ is étale.
Let $X° := γ^{-1}(Y°)$ denote the preimage.  Choose a closed point $y ∈ Y°$ and
a closed point $x$ of the fibre $X_y := γ^{-1}(y)$.  The étale map $γ|_{X°} : X°
→ Y°$ corresponds to a transitive action of the étale fundamental group
$\what{π} := \what{π}_1(Y°, y)$ on $X_y$.  The group $\what{π}$ also acts on
$$
G := \img \bigl( \what{π} → \text{Permutation group of } X_y \bigr)
$$
by left multiplication.  We obtain a sequence of finite sets with transitive
$\what{π}$-action and a corresponding sequence of étale covers,
$$
\xymatrix{ %
  G \ar[rr]^{g \mapsto g· x} && X_y \ar[r]^{\vphantom{p}\text{const.}} & \{ y \} %
  &\text{and}& \wtilde X° \ar[r]^{\wtilde{γ}°} & X° \ar[r]^{γ|_{X°}} & Y°, %
}
$$
where $Γ° := (γ|_{X°}) ◦ \wtilde{γ}°$ is Galois with group $G$ and $\wtilde{γ}°$
is Galois with Galois group equal to the isotropy group $H := \Iso_x \leqslant
G$ of the point $x ∈ X_y$, respectively.  Applying Theorem~\ref{thm:ZMT-eq} to
the quasi-finite morphism $\wtilde X° → X$, we obtain a normal, quasi-projective
variety $\wtilde X$ and a diagram
$$
\xymatrix{ %
  \wtilde X° \ar[rr]^{\vphantom{\wtilde{γ}}\text{immersion}} && \wtilde X \ar[rrrr]^{\wtilde{γ}\text{, Galois with group $H$}} &&&& X \ar[rr]^{\vphantom{\wtilde{γ}}γ\text{, finite}} && Y.
}
$$

\subsection*{Step 2: Verification of \ref*{il:GA1}}

We need to show that $Γ := γ◦\wtilde{γ}$ is Galois with group $G$.  Applying
Theorem~\ref{thm:ZMT-eq} to the quasi-finite morphism $\wtilde X° → Y$, we
obtain a normal, quasi-projective variety $\wtilde X'$ and morphisms $\wtilde X°
\into \wtilde X' \xrightarrow{Γ'} Y$, where $Γ'$ is Galois with group $G$.
Since $Γ$ and $Γ'$ are both finite morphisms, the uniqueness statement of
Theorem~\ref{thm:ZMT-eq} shows that $\wtilde X = \wtilde X'$ and $Γ = Γ'$.  This
establishes \ref{il:GA1}.

\subsection*{Step 3: Verification of \ref*{il:GA2}}
\approvals{Greb & yes \\ Kebekus & yes \\ Peternell & yes}

As $Γ = γ ◦ \wtilde{γ}$, it follows from \cite[I.~Thm.~10.11]{SGA1} that $\Br(Γ)
⊇ \Br(γ)$.  On the other hand, the construction immediately shows that
$(γ|_{X°})◦\wtilde{γ}°$ is étale, so that $\Br(Γ) ⊆ \Br(γ)$.  The branch loci of
$γ$ and $Γ$ therefore coincide, as claimed in \ref{il:GA2}.  This finishes the
proof of Theorem~\ref{thm:galoisClosure}.  \qed

}

\end{document}